\documentclass{article}

\author{Ofir Gorodetsky} \title{The variance of integers without small prime factors in short intervals}
\date{}

\usepackage[margin=1in]{geometry}
\usepackage[all]{xy}
\usepackage{amssymb}
\usepackage{amsfonts}
\usepackage{amsthm}
\usepackage{amsmath}
\usepackage{hyperref}
\usepackage{mathtools}
\usepackage{url}
\usepackage{color}

\mathtoolsset{showonlyrefs}

\theoremstyle{plain}
\newtheorem*{thm*}{Theorem}
\newtheorem{thm}{Theorem}[section]

\newtheorem{lem}[thm]{Lemma}  
\newtheorem{proposition}[thm]{Proposition}
\newtheorem{cor}[thm]{Corollary}

\newtheorem{conj}{Conjecture} 

\theoremstyle{remark}
\newtheorem{remark}{Remark}

\newcommand{\RR}{\mathbb{R}}

\newcommand{\NN}{\mathbb{N}}
\newcommand{\ZZ}{\mathbb{Z}}

\newcommand{\CC}{\mathbb{C}}

\newcommand{\PrimesInd}{\mathbf{1}_{\mathbb{P}}}	
\newcommand{\Prob}{P}

\newcommand{\Psiomega}{\Psi_{d_2\mu^2}}

\numberwithin{equation}{section}

\newcommand{\Addresses}{{
		\bigskip
		\footnotesize
		
		\textsc{Department of Mathematics, Technion - Israel Institute of Technology, Haifa, Israel}\par\nopagebreak
		\textit{E-mail address:} \texttt{ofir.gor@technion.ac.il}
}}

\begin{document}
	
	\maketitle
	\begin{abstract}
		The variance of primes in short intervals relates to the Riemann Hypothesis, Montgomery's Pair Correlation Conjecture and the Hardy--Littlewood Conjecture. In regards to its asymptotics, very little is known unconditionally. We study the variance of integers without prime factors below $y$, in short intervals. We use complex analysis and sieve theory to prove an unconditional asymptotic result in a range for which we give evidence is qualitatively best possible. We find that this variance connects with statistics of $y$-smooth numbers, and, as with primes, is asymptotically smaller than the naive probabilistic prediction once the length of the interval is at least a power of $y$.
	\end{abstract}
	\section{Introduction}
	\subsection{The variance of primes in short intervals}
	 Saffari and Vaughan \cite[Lem.~7]{Saffari1977}, building on pioneering work of Selberg \cite[pp.~160-178]{selberg}, studied the variance of primes in short intervals. They showed that, conditionally on the Riemann Hypothesis (RH), one has
		\[  \frac{1}{X} \int_{0}^{X} \big( \sum_{x<n\le x+H}\Lambda(n) -H \big)^2 {\rm d}x \ll H (1 + \log (X/H))^2\]
		for $1 \le H\le X$, which implies that 
\begin{equation}\label{eq:mont conj}
\sum_{n \in (x,x+H]} \Lambda(n) = H + O_{\varepsilon}(x^{\varepsilon}H^{1/2})
\end{equation}
holds for almost all $x \in [1,X]$, as $X \to \infty$. This is one motivation for a conjecture of Montgomery and Vaughan, saying \eqref{eq:mont conj} holds for \emph{all} $1 \le H \le x$ \cite[Conj.~13.4]{Montgomery2007}.
Goldston and Montgomery showed \cite[Cor.]{Goldston1987} -- on RH -- that
	\begin{equation}\label{eq:var von}
	\frac{1}{X} \int_{0}^{X}\big( \sum_{x < n \le x+H} \Lambda(n) - H\big)^2 {\rm d}x \sim H \log X \big(1  - \frac{\log H}{\log X}\big),\qquad X \to \infty
\end{equation}
	uniformly for $1 \le H \le X^{1-\varepsilon}$ is equivalent to the Strong Pair Correlation Conjecture for the Riemann zeta function. This is particularly interesting as the Cram\'er model predicts a larger asymptotic, omitting the $-\log H/\log X$ term in \eqref{eq:var von}, a term which becomes significant once $H$ is at least a power of $X$. This term can also be detected, in some range of $H$, assuming RH and the Hardy--Littlewood Conjecture \cite{Montgomery2002,Montgomery2004}. 
	The conjectured estimate \eqref{eq:var von} implies that almost all intervals $[x,x+\omega(x)\log x]$ contain $\omega(x)(1+o(1))$ primes, where $\omega(x)=o(x)$ is any function tending to $\infty$ as $x \to \infty$.
	
	Let $\PrimesInd$ be the indicator of primes. All of these results have equivalent analogues with $\PrimesInd$ instead of $\Lambda$, for instance
	\begin{equation}\label{eq:var primes}
		\frac{1}{X} \int_{2}^{X}\bigg( \sum_{x < n \le x+H} \PrimesInd(n) - \int_{x}^{x+H} \frac{{\rm d}t}{\log t}\bigg)^2 {\rm d}x \sim \frac{H}{\log X} \big(1  - \frac{\log H}{\log X}\big)
	\end{equation}
	is equivalent to \eqref{eq:var von}, see Proposition 2 in the appendix of \cite{Kuperberg2020} (cf.~\cite{Montgomery2010}).
	\subsection{Integers without small prime factors}
	Our aim in this work is to obtain unconditional results about the variance in short intervals of integers without small prime factors, known in the literature as rough numbers or sifted numbers. The motivation is that rough numbers generalize primes. Indeed, integers with no prime factors $\le \sqrt{x}$ coincide with primes in the interval $(\sqrt{x},x]$. Hence it would be interesting to see how large can we take $y$ as a function of $H$ (or alternatively, how large we can take $H$ to be given $y$) and still obtain an asymptotic result for the corresponding variance.
	Let $y\ge 2$ be a parameter (that might depend eventually on $H$ and $X$). We let $\alpha_y$ be the indicator of integers without prime factors $\le y$: 
	\[ \alpha_y(n) = \begin{cases} 1 & \text{if every prime dividing }n\text{ is}> y,\\ 0& \text{otherwise.}\end{cases}\]
For fixed $y$, $\alpha_y$ is a periodic function whose period is the primorial \[q_y:=\prod_{p \le y} p = e^{y(1+o(1))}\] by the Chinese remainder theorem, and its mean value is 
	\begin{equation}\label{eq:mertens3}
		\Prob_y := \prod_{p \le y} \left(1-\frac{1}{p}\right) = \frac{e^{-\gamma}}{\log y}\left(1+O\left(\frac{1}{\log y}\right)\right)
	\end{equation}
	by Mertens' theorem, where throughout $\gamma$ is the Euler--Mascheroni constant. Our object of study is
	\[ V(X,H,y) := \frac{1}{X}\int_{0}^{X} \big(\sum_{x<n\le x+H} \alpha_y(n) - H\Prob_y\big)^2 {\rm d}x.\]
	 Throughout, $y$, $H$ and $X$ will be integer parameters\footnote{As our parameters are integers, the continuous variance $V(X,H,y)$ coincides with the discrete variance $\sum_{n=0}^{X-1} (\sum_{i=1}^{H}\alpha(n+i)-H\Prob_y)^2/X$.} tending together to $\infty$. As $\alpha_y$ is $q_y$-periodic, \[V(X,H,y) = V(X,H-q_y,y)\]
	 if $H>q_y$. Hence we may assume without loss of generality that $H \le q_y$, implying $\limsup (\log H)/y \le 1$ by the Prime Number Theorem. Our main result will hold under a condition which is just slightly more restrictive, namely
\[y \ge (2+\varepsilon)\log H\]
for some (arbitrary) $\varepsilon>0$. 	

If one is not interested in asymptotics of the variance, but in bounds that yield almost-all results on $\alpha_y$ in short intervals, then extremely strong results were obtained by Friedlander and Iwaniec \cite[Ch.~6.10]{Friedlander2010}. Instead of studying the variance of $\alpha_y$ asymptotically, they upper bound the variance of (upper-bound and lower-bound) sieve weights in short intervals \cite[Prop.~6.26]{Friedlander2010}. From this they deduce that, if $y \le X^{1/20}$ and $H/\log y \to \infty$, then $\sum_{x<n\le x+H}\alpha_y(n)\asymp H\Prob_y$ for almost all $x \in [X,2X]$  \cite[Cor.~6.28]{Friedlander2010}.
\subsection{Main result}
The main term that comes out from our analysis of $V(X,H,y)$ is a (somewhat complicated) function $M(H,y)$ of $H$ and $y$ only. It is given as the following finite product times a finite sum:
\begin{equation}\label{eq:mt}
	M(H,y) := \prod_{2<p \le y} \left(1- \frac{2}{p}\right) \sum_{n \ge 1} g_y(n) \left\{ \frac{H}{2n}\right\} \left(1- \left\{ \frac{H}{2n}\right\} \right),
\end{equation}
where $\{ x \}$ is the fractional part of $x$ and $g_y$ is the multiplicative function supported on squarefrees and given on primes by
\begin{equation}\label{eq:gydef}
	g_y(p) = \begin{cases} \frac{p}{p-2} & \text{if }2<p \le y,\\0 & \text{otherwise.}\end{cases}
\end{equation}
In \S\ref{sec:mt} we describe the asymptotics of $M(H,y)$. We are ready to state our main result.
\begin{thm}\label{thm:ranges}
Fix $\varepsilon\in (0,1)$. Suppose the parameters $y,H\le X$ tend to $\infty$ in such a way that $y \ge (2+\varepsilon)\log H$ holds. Set $a=\log \log H/\log y \in (0,  1)$. If
\begin{equation}\label{eq:sievecond}
(1+ a)\frac{\log H}{\log \log H} \le (1-\varepsilon)\frac{\log X}{\log y},
\end{equation}
 then, as $H\to \infty$,
\begin{equation}\label{eq:Vasymp}
 V(X,H,y) \sim M(H,y).
\end{equation}
\end{thm}
For example,
\begin{itemize}
	\item If we take $H=X^{A}$ for fixed $A\in (0,1/2)$, then \eqref{eq:Vasymp} holds for $(2A+\varepsilon)\log X \le y \le (\log X)^{\frac{1}{A}-1-\varepsilon}$.
	\item If we take $H=\exp((\log X)^A)$ for $A\in(0,1)$, then \eqref{eq:Vasymp} holds for $(2+\varepsilon)(\log X)^A \le y \le \exp((\log X)^{1-A})$.
\end{itemize}
The naive prediction for $V(X,H,y)$, based on modeling $(\alpha_y(n))_{n}$ by independent Bernoulli random variables with success probability $\Prob_y$, is $H\Prob_y$. By Theorem~\ref{thm:mt} below, $M(H,y)$ is asymptotic to $H\Prob_y$ as long as $H=y^{o(1)}$. However, for $H$ which is at least a power of $y$ (namely $H \ge y^{\varepsilon}$ for some $\varepsilon>0$), Theorem~\ref{thm:mt} below implies that $M(H,y)$ is asymptotically smaller than this prediction, in the same way the variance of primes is once $H$ is a power of $X$.

The range in Theorem~\ref{thm:ranges} gets worse as $H$ and $y$ increase. This is to be expected. Larger $H$ means we have less averaging present in our variance, where we think of $[0,X]$ as a union of $X/H$ disjoint intervals of size $H$. This is similar to how the results in \cite{Hall1982,Gorodetsky2021,Montgomery2004} on squarefrees and on primes hold only for $H$ sufficiently small. Larger $y$ means $\alpha_y$ becomes `closer' to the indicator of primes, and so we expect our range to deteriorate as $y$ grows, in the same way results in sieve theory deteriorate with $y$.

The following is an immediate corollary of Theorem~\ref{thm:ranges}.
\begin{cor}
Fix $\varepsilon>0$. Let $w(x)$ be a function tending to $\infty$. Given $y,H\le X$ satisfying $y\ge (2+\varepsilon)\log H$ and \eqref{eq:sievecond}, we have, for almost all $x \in [0,X]$,
	\[ \sum_{x<n \le x+H} \alpha_y(n) = H \Prob_y + O_{\varepsilon}\left( M(H,y)^{1/2} w(x)\right).\]
\end{cor}
Recall that an integer is called $y$-smooth if it is indivisible by primes greater than $y$. We denote by  $\Psi(x,y)$ the number of $y$-smooth integers up to $x$. Theorem~\ref{thm:mt} below tells us that $M(H,y) \ll_{\varepsilon} \Psi(H,y)$ holds for $y \ge (2+\varepsilon)\log H$. We conjecture the following, in the spirit of the conjecture \cite[Conj.~13.4]{Montgomery2007} for primes.
\begin{conj}
Fix $\varepsilon>0$.	For all $2\le y,H\le X$ satisfying $y \ge (2+\varepsilon)\log H$ and \eqref{eq:sievecond}, we have
	\[ \sum_{X<n \le X+H} \alpha_y(n) = H \Prob_y + O_{\varepsilon}\left( \min\{H,y\}^{\varepsilon}\Psi(H,y)^{1/2}\right).\]	
\end{conj}

\subsection{Estimates of the main term}\label{sec:mt}
For a given arithmetic function $f$ we let
\begin{equation}\label{eq:psif}
\Psi_{f}(x,y) = \sum_{\substack{n \le x \\ n \text{ is }y\text{-smooth}}}f(n).
\end{equation}
It turns out $M(H,y)$ is related to $\Psi_{\mu^2}(H,y)$, the number of  $y$-smooth squarefree integers up to $H$. This is because $g_y$ is supported on such numbers.  We define 
\begin{equation}\label{eq:deflam}
	\lambda(u):=e^{-\gamma}\int_{u}^{\infty}\rho(v){\rm d}v
\end{equation}
for $u \ge 0$ where $\rho$ is the Dickman function. The function $\lambda$ is implicit in works of van Lint and Richert \cite{vLR} and de Bruijn and van Lint \cite{dBL}, e.g.~\cite[Eq.~(1.3)]{dBL} says that for fixed $u\ge 0$, 
\[ \sum_{n \ge y^u,\, n\text{ is }y\text{-smooth}} 1/n \sim e^{\gamma}\lambda(u)\log y\]
as $y \to \infty$. The asymptotics of $\lambda$ are well understood: by \cite[Eq.~(2.6)]{Tenenbaum2008},
	\begin{equation}\label{eq:lambasym}
		\lambda(u) = e^{-\gamma}\frac{\rho(u)}{\xi(u)}\left(1+O(u^{-1})\right)
	\end{equation}
for $u>0$, where $\xi(u)$ is the unique solution to $e^{\xi(u)}=1+u\xi(u)$ ($\xi(u)\sim \log u$ as $u \to \infty$ by  Lemma~\ref{lem:hildlem1}). For $u \in [0,1]$,
\begin{equation}\label{eq:lam01}
\lambda(u)=1-e^{-\gamma}u.
\end{equation}
Indeed, since $\rho(v)=1$ in $[0,1]$, \eqref{eq:lam01} is equivalent to $\lambda(0)=1$, which is well known \cite[Thm.~2.2]{dBL}. Let $\alpha=\alpha(H/2,y)>0$  be the solution to
\begin{equation}\label{eq:alphadef}
	\log (H/2) = \sum_{p \le y} \frac{\log p}{p^{\alpha}+1};
\end{equation}
the size of $\alpha$ is described in Lemma~\ref{lem:hild2}.
We are ready to state our theorem on $M(H,y)$.
\begin{thm}\label{thm:mt}
	Suppose $y\ge 2$, $H\ge 2$. Denote $u=\log H/\log y$. Fix $\varepsilon>0$. The following estimates hold.
	\begin{enumerate}
	\item Suppose $y \ge \exp((\log \log H)^{5/3+\varepsilon})$. As $H \to \infty$, we have
	\begin{equation}\label{eq:M2} M(H,y) \sim H \Prob_y \lambda(u).
\end{equation}
\item Suppose $H \ge y \ge (\log H)^{2+\varepsilon}$. As $\min\{u,\log y/\log \log H\} \to \infty$, we have	\begin{equation}\label{eq:M5}
	M(H,y)\sim \frac{e^{-\gamma}}{\xi(u)}\Prob_y \Psi(H,y).
\end{equation}
\item Suppose $H \ge y \ge (2+\varepsilon)\log H$ and $H \ge C_{\varepsilon}$. As $u \to \infty$, we have	\begin{equation}\label{eq:M4}
		M(H,y)\sim\frac{2\zeta(\alpha-1)}{\alpha-1} \prod_{p \le y} \left(1-\frac{2}{p(1+p^{-\alpha})} \right) \Psi_{\mu^2}(H,y).
	\end{equation}
\item Uniformly for $y \ge (\log H)^{1+\varepsilon}$ and $H \ge C_{\varepsilon}$, 
\begin{equation}\label{eq:last}
M(H,y) \asymp_{\varepsilon} \Prob_y \Psi_{\mu^2}(H,y)/\log(2+u).
\end{equation}
\end{enumerate}
\end{thm}
\begin{remark}
Estimate~\eqref{eq:M2} shows that in a wide range $M(H,y)$ is asymptotically of the form $H\Prob_y$ (the naive expectation) times a decreasing function of $\log H/\log y$, much like \eqref{eq:var primes}.
\end{remark}
\begin{remark}
Theorem~\ref{thm:mt}'s proof supplies rates of convergence for \eqref{eq:M2}--\eqref{eq:M4}, see \eqref{eq:MHstr}, \eqref{eq:M5form} and \eqref{eq:M5form2}.
\end{remark}
To make sense of Theorem~\ref{thm:mt}, we recall basic facts about $\Psi(x,y)$ and $\Psi_{\mu^2}(x,y)$.
For any $\varepsilon>0$, Hildebrand \cite{Hildebrand19862} proved that, in the range $x\ge y\ge \exp((\log \log x)^{5/3+\varepsilon})$, 
\begin{equation}\label{eq:psirho}
	\Psi(x,y)=x\rho\left(\frac{\log x}{\log y}\right)\bigg(1+O_{\varepsilon}\bigg( \frac{\log \left(\frac{\log x}{\log y}+1\right)}{\log y}\bigg)\bigg).
\end{equation}
Moreover, \eqref{eq:psirho} persists for $x\ge y \ge (\log x)^{2+\varepsilon}$ under RH \cite{Hildebrand1984}.
Ivi\'c and Tenenbaum proved \cite[Eq.~(2.12)]{Ivic1986} 
\begin{equation}\label{eq:ivict}
\Psi_{\mu^2}(x,y)= \frac{\Psi(x,y)}{\zeta(2)}\bigg(1+O_{\varepsilon}\bigg( \frac{\log \log (y+1)}{\log y}\bigg)\bigg)
\end{equation}
holds for $x\ge y \ge (\log x)^{2+\varepsilon}$ (cf.~Naimi \cite{Naimi1988}).
By \cite{deBruijn1966} and \cite{Granville1989},
\begin{equation}\label{eq:psimu2crude}
	\ \Psi(x,y),	\Psi_{\mu^2}(x,y) = x^{1-\frac{\log \log x}{\log y}+o(1)}
\end{equation}
hold when $y \ge \log x$. 
\subsection{Structure of paper}
The proof of Theorem~\ref{thm:ranges} relies crucially on the fundamental lemma of the combinatorial sieve. Particularly, we need its following corollary, applied to the set $A=\{\prod_{i=1}^{k}(n+h_i) : n \in I\}$ where $I$ is an interval and $h_i$ are nonnegative integers. We apply the fundamental lemma as appearing in Friedlander and Iwaniec  \cite[Thm.~6.12]{Friedlander2010}. Weaker versions go back to the works of Halberstam and Richert \cite{HR1,HR2}.
\begin{lem}\label{lem:sievecomb}
	Let $I$ be an interval of length $X$. Fix $k\ge 1$ and let $h_1,\ldots,h_k$ be nonnegative integers. For all $y,u \ge 1$ we have \[\sum_{n \in I\cap \ZZ} \prod_{i=1}^{k}\alpha_y(n+h_i) = X \prod_{p \le y} \left( 1- \frac{\nu_p(\mathbf{h})}{p}\right) \left( 1+O_k\left(u^{-u(1+o(1))}\right) \right) +  O_k\bigg( \Psi_{d_k\mu^2}(y^u,y)\bigg)\]
	where $\omega(d)$ is the number of prime factors of $d$, $\nu_p(\mathbf{h}) = \#\{ h_1 \bmod p, h_2 \bmod p ,\ldots, h_k \bmod p\}$, $d_k$ is the $k$-fold divisor function and we use the notation \eqref{eq:psif}. 
\end{lem}
In \S\ref{sec:sieve} we expand the square in $V(X,H,y)$, rearrange and apply Lemma~\ref{lem:sievecomb} with $k=1,2$. In this way, the main term $M(H,y)$ quickly appears and we are able to bound
\[ V(X,H,y) - M(H,y),\]
see Proposition~\ref{prop:sieveapproach}.

In \S\ref{sec:largey} we lay the ground for estimating $M(H,y)$ for large $y$ (namely $y \ge \exp(C(\log H)^{1/2}(\log \log H)^{1/2})$). This is done by working in physical space and is relatively elementary.

In \S\ref{sec:saddle}, the most technical section, we  lay the ground for studying $M(H,y)$ for smaller $y$. This relies on the saddle point method, originally developed in \cite{Hildebrand1986}.

In \S\ref{sec:final} we conclude the proofs of Theorems~\ref{thm:ranges} and \ref{thm:mt}. We use the estimates in \S\ref{sec:largey}-\S\ref{sec:saddle} to prove Theorem~\ref{thm:mt}. Once we have Theorem~\ref{thm:mt} available to us, we know the growth rate of $M(H,y)$ and can deduce from our bound on $V(X,H,y) - M(H,y)$ the asymptotic result $V(X,H,y) \sim M(H,y)$ in a certain range.

The proofs of Theorems~\ref{thm:ranges} and \ref{thm:mt} are based on ideas from Hall's work on the variance of squarefrees in short intervals \cite{Hall1982}. 

Lemma~\ref{lem:sievecomb} with $k \ge 3$ paves the way to studying higher moments of $\sum_{x<n\le x+H}\alpha_y(n)- H\Prob_y$, and we hope to study such moments in future work.
\subsection{Discussion}
\subsubsection{The case of \texorpdfstring{$X$}{X} being a primorial}
	As far as we are aware, asymptotics of  $V(X,H,y)$ were not studied previously, except in a very particular case which we now describe. If $X$ is exactly the primorial $q_y$, then $\alpha_y$ is $X$-periodic and may be considered as a function on $\ZZ/X\ZZ$, and 
\[ V(q_y,H,y) = \frac{1}{q_y} \sum_{i \in \ZZ/q_y\ZZ}\left( \sum_{n=1}^{H} \alpha_y(n+i) - H\frac{\phi(q_y)}{q_y}\right)^2.\]
This discrete sum was studied by Hausman and Shapiro \cite{Hausman1973} and by Montgomery and Vaughan \cite{Montgomery1986}. These authors studied, more generally, the quantity 
\[ V_q(H):=\frac{1}{q} \sum_{i \in \ZZ/q\ZZ} \left( \sum_{\substack{n=1\\ (n+i,q)=1}}^{H} 1 - H\frac{\phi(q)}{q}\right)^2\]
for squarefree $q$. Hausman and Shapiro expressed $V_q$ as
\begin{equation}\label{eq:vqh}
V_q(H) = \frac{\phi(q)^2}{q^2}\sum_{r \mid q} \prod_{\substack{p \mid q \\ p \nmid r}} \frac{p(p-2)}{(p-1)^2}r^2 \phi(r)^{-2} \left\{ \frac{H}{r}\right\} \left(1-\left\{ \frac{H}{r}\right\}\right)
\end{equation}
and Montgomery and Vaughan used this expression to show that
\[ H\frac{\phi(q)}{q} \left(1-\prod_{p \mid q,\, p>H}\left(1-\frac{1}{p}\right) +O\left(\frac{\phi(q)}{q}\right) \right) \le V_q(H)\le H \frac{\phi(q)}{q}.\]
In the case of $q=q_y$, these bounds become
\[ H\Prob_y \left(1- \prod_{H<p\le y}\left(1-\frac{1}{p}\right) + O\left( \frac{1}{\log y}\right)\right) \le V(q_y,H,y) \le H\Prob_y\]
by \eqref{eq:mertens3}. If $H=y^{o(1)}$ this leads to the asymptotic formula $V(q_y,H,y) \sim H\Prob_y$, in agreement with our results. However, this argument does not show that $V(q_y,H,y) \sim H\Prob_y$ fails to hold if $H\ge y^{\varepsilon}$ for some $\varepsilon>0$. Theorem~\ref{thm:mt} does show that, using the following identity, which follows easily from \eqref{eq:vqh}.
\begin{lem}\label{lem:primorial}
We have $V_{q_y}(H)= M(H,y)$.
\end{lem}
\begin{proof}
Formula~\eqref{eq:vqh} with $q=q_y$ shows that
\[	V(q_y,H,y)= \prod_{p\le y} \left(1-\frac{1}{p}\right)^2 \sum_{\substack{r\ge 1\\ p \mid r \Rightarrow p \le y}}  \mu^2(r) r^2 \phi(r)^{-2} \prod_{\substack{p \le y\\ p \nmid r}}\frac{p(p-2)}{(p-1)^2} \left\{ \frac{H}{r}\right\} \left(1-\left\{ \frac{H}{r}\right\}\right).\]
If $r$ is odd, then its contribution is $0$ because of the term $p(p-2)/(p-1)^2$ vanishing at $p=2$. Thus we may assume $r$ is even. So we replace $r$ by $2r$, obtaining
\begin{align}\label{eq:vqh3}
	V(q_y,H,y)&=  \prod_{2<p\le y} \left(1-\frac{2}{p}\right) \sum_{\substack{r \text{ odd}\\ p \mid r \Rightarrow p \le y}}  \mu^2(r) r^2 \phi(r)^{-2} \prod_{\substack{p \le y\\ p \mid r}}\frac{(p-1)^2}{p(p-2)} \left\{ \frac{H}{2r}\right\} \left(1-\left\{ \frac{H}{2r}\right\}\right)\\
	&= \prod_{2<p\le y} \left(1-\frac{2}{p}\right) \sum_{r \ge 1}  g_y(r)\left\{ \frac{H}{2r}\right\} \left(1-\left\{ \frac{H}{2r}\right\}\right).
\end{align}
This coincides with $M(H,y)$, defined in \eqref{eq:mt}.
\end{proof}
\subsubsection{Fluctuations and limitations}\label{sec:limitations}
Let $X \ge y\ge 2$ and fix $\varepsilon>0$. Friedlander, Granville, Hildebrand and Maier \cite[Thm.~B1]{Friedlander1991} proved the following. If $y \ge (\log X)^{6/5+\varepsilon}$ and $v=\log X/\log y \ge C_{\varepsilon}$, then there exist 
\[\min\left\{ \frac{y}{2}, y^{1-\frac{2}{v+2}}\right\} \le y' \le y, \qquad X\left(1-\frac{1}{\log X}\right) \le X' \le X\]
so that
\begin{equation}\label{eq:osc}
\left| \frac{1}{X'}\sum_{1 \le n \le X'} \alpha_{y'}(n) - \Prob_{y'} \right| \ge \Prob_{y'} v^{-v(1+o(1))}.
\end{equation}
Let $H  \in [y,X]$. By Cauchy--Schwarz and Corollary~\ref{cor:expectation} with $u=\log H/\log y' \ge 1$ we have
\begin{equation}\label{eq:Hlarge} V(X',H,y') \ge H^2 \left( \frac{1}{X'}\sum_{1 \le n \le X'} \alpha_{y'}(n) - \Prob_{y'} + O\left( \frac{H^{1-\frac{\log \log H}{\log y'}+o(1)}}{X'}\right)\right)^2,
\end{equation}
where we made use of \eqref{eq:psimu2crude} (and our assumptions on $y$ and $v$).
If the error term in \eqref{eq:Hlarge} is negligible compared to $\Prob_{y'}v^{-v(1+o(1))}$ (e.g.~if $(1-\log \log H/\log y') \log H \le (1-\varepsilon)(1- \log \log X/\log y)\log X$), then it follows by \eqref{eq:osc} that	
\[ V(X',H,y') \gg H^2 \Prob_{y'}^2 v^{-2v(1+o(1))}\]
for some $X'$ and $y'$ close to the original $X$ and $y$. By \eqref{eq:psimu2crude} this lower bound exceeds the main term $\Prob_{y'} \Psi_{\mu^2}(H,y')^{1+o(1)}$ if 
\[	\left(1+ \frac{\log \log H}{\log y'}\right)\frac{\log H}{\log \log H} \ge (2+\varepsilon)\frac{\log X}{\log y'}.\]
This shows that in general we cannot replace the constant $1-\varepsilon$ in \eqref{eq:sievecond} by $2+\varepsilon$. We conjecture that we are off by a factor of $2$, that is, that $2-\varepsilon$ is the best possible constant.
\begin{conj}\label{conj:pred}
Fix $\varepsilon\in (0,1)$. Suppose the parameters $y,H\le X$ tend to $\infty$ in such a way that  $y \ge (2+\varepsilon)\log H$ holds. Set $a=\log \log H/\log y \in (0,  1)$. If
\[	(1+ a)\frac{\log H}{\log \log H} \le (2-\varepsilon)\frac{\log X}{\log y},\]
then 
\[ V(X,H,y) \sim M(H,y).\]
\end{conj}
If one replaces the main term $H\Prob_y$ appearing in the definition of $V(X,H,y)$ (which is currently independent of $x$, the starting point of the interval $(x,x+H]$) by a more nuanced approximation for $\sum_{n\in (x,x+H]} \alpha_y(n)$ we expect our asymptotic result to hold in a range going beyond Conjecture~\ref{conj:pred}, but are not able to show this. Two examples of less naive main terms are $H\sum_{n \le X} \alpha_y(n)/X$, and de Bruijn's approximation \cite{DeBruijn1950} to $\sum_{n \le x+H} \alpha_y(n)-\sum_{n \le x} \alpha_y(n)$, namely 
\[ (x+H)\Prob_y e^{\gamma}\log y \int_{0}^{\infty}\omega\left(\frac{\log (x+H)}{\log y}-v\right)y^{-v}{\rm d}v-x\Prob_y e^{\gamma}\log y \int_{0}^{\infty}\omega\left(\frac{\log x}{\log y}-v\right)y^{-v}{\rm d}v\] 
where $\omega$ is the Buchstab function.

Granville and Soundararajan studied fluctuations of $\alpha_y$ in short intervals \cite[Cor.~1.2]{Granville2007}. Concretely, given $y\ge C$ and $y^{c_1} \ge u \ge C$ they proved the existence of intervals $I_{\pm}$ of length at least $y^{2c_1u}$ such that $\sum_{n \in I_+} \alpha_y(n)-H\Prob_y \ge H\Prob_y u^{-uc_2}$ and $\sum_{n \in I_+} \alpha_y(n)-H\Prob_y \le -H\Prob_y u^{-uc_2}$. Here $C$, $c_1$ and $c_2$ are absolute constants.
\subsubsection{Duality}
Our results show a duality between rough numbers and smooth numbers. Such a connection is to be expected, because we may write $\alpha_y$ as $1*\beta_y$ where $\beta_y$ is supported on $y$-smooth numbers:
\[ \beta_y(n) = \mu(n) \mathbf{1}_{n \text{ is }y\text{-smooth}}.\]
As a consequence, $\sum_{n \le x} \alpha_y(n) = \sum_{d \text{ is }y\text{-smooth}} \mu(d) \lfloor x/d \rfloor$ and this may be used to obtain \cite[Thm.~III.6.2]{Tenenbaum2015}
\[\sum_{n \le x} \alpha_y(n) = x\Prob_y + O(\Psi(x,y)).\]
See \cite[Ch.~III.6]{Tenenbaum2015} for further examples of this duality.
\subsection{Conventions}
Throughout, $C$ and $c$ are absolute positive constants that may change between occurrences. Similarly, $C_{\varepsilon}$ and $c_{\varepsilon}$ are positive functions of $\varepsilon$  that may change between occurrences.
\section{Application of sieve theory}\label{sec:sieve}
The following corollary and proposition are consequences of Lemma~\ref{lem:sievecomb}.
\begin{cor}\label{cor:expectation}
	Let $1 \le H\le X$ be integers. Given an integer $y \ge 2$ and a real parameter $u \ge 1$ we have \begin{equation}\label{eq:expectation}
		\frac{1}{X} \int_{0}^{X} \bigg( \sum_{x<n\le x+ H} \alpha_y(n)\bigg){\rm d}x = \frac{H}{X} \sum_{n=1}^{X} \alpha_y(n) + O\left( \frac{H}{X} (H \Prob_y u^{-u(1+o(1))} + \Psi_{\mu^2}(y^u,y) )\right).
	\end{equation}
\end{cor}
\begin{proof}
The left-hand side is
\[ \frac{1}{X}\sum_{j=1}^{X}\sum_{i=0}^{H-1}\alpha_{y}(j+i) = \frac{H}{X}\sum_{n=1}^{X} \alpha_y(n) - \frac{1}{X}\sum_{j=1}^{H} \alpha_y(j)(H-j) + \frac{1}{X}\sum_{j=X+1}^{X+H} \alpha_y(j) (X+H-j).\]
The error term arises from bounding the two $j$-sums. By summation by parts they can be written as
\[-\frac{1}{X}\sum_{j=1}^{H} \alpha_y(j)(H-j) + \frac{1}{X}\sum_{j=X+1}^{X+H} \alpha_y(j) (X+H-j)
=\frac{1}{X}\sum_{j=1}^{H-1} \bigg(\sum_{1 \le n \le j} \alpha_y(n+X) -  \sum_{1 \le n \le j} \alpha_y(n)\bigg). \]
Applying Lemma~\ref{lem:sievecomb} with $k=1$, they contribute
\[ \ll \frac{H}{X}\left( H\Prob_y u^{-u(1+o(1))} + \Psi_{\mu^2}(y^u,y)\right)\]
and the claim is established.
\end{proof}

	\begin{proposition}\label{prop:sieveapproach}
Let $1 \le H \le X$ be integers. Given an integer $y\ge 2$ and a real parameter $u \ge 1$ we have
\[	V(X,H,y) = M(H,y)+ O\left( H^2\left( \Prob_y  u^{-u(1+o(1))}+\frac{ \Psiomega(y^u,y)}{X}\right)\right).\]
	\end{proposition}
	\begin{proof}
	Expanding the square in $V(X,H,y)$ we have
		\[ V(X,H,y) = V_1 + V_2 + V_3 \]
		where
	\[ V_1 = \frac{1}{X} \int_{0}^{X} \bigg( \sum_{x<n\le x+H} \alpha_{y}(n)\bigg)^2 {\rm d}x, \qquad V_2 = H^2 \Prob^2_y,\qquad V_3 = - 2H\Prob_y \frac{1}{X} \int_{0}^{X} \bigg( \sum_{x<n \le x+H}\alpha_{y}(n)\bigg){\rm d}x.\]
We start by studying $V_3$. By Corollary~\ref{cor:expectation} and Lemma~\ref{lem:sievecomb} with $k=1$,
\begin{align}\label{eq:diag}
\frac{1}{X} \int_{0}^{X} \bigg( \sum_{x<n \le x+H}\alpha_{y}(n)\bigg){\rm d}x &= \frac{H}{X}\sum_{n=1}^{X} \alpha_y(n) + O\left( \frac{H}{X} (H\Prob_y u^{-u(1+o(1))}+\Psi_{\mu^2}(y^u,y))\right)\\
&=  H \Prob_y \left( 1 + O\left( u^{-u(1+o(1))}\right) \right) + O\left( \frac{H}{X}\Psi_{\mu^2}(y^u,y)\right).
\end{align}
		Hence
		\[ V_3 = -2H^2\Prob_y^2 + O\left( H^2 \Prob_y^2 u^{-u(1+o(1))}+\frac{H^2 \Prob_y \Psi_{\mu^2}(y^u,y)}{X}\right).\]
		We turn to $V_1$. We separate it into diagonal and off-diagonal parts:
\[ V_1  = \frac{1}{X} \sum_{n=1}^{X}\left(\sum_{i=0}^{H-1} \alpha_y(n+i)\right)^2 = \sum_{i=0}^{H-1}\frac{1}{X} \sum_{n=1}^{X} \alpha_y(n+i) + 2\sum_{0 \le i<j \le H-1} \frac{1}{X} \sum_{n=1}^{X} \alpha_y(n+i) \alpha_y(n+j) .\]
		The diagonal part was already estimated in \eqref{eq:diag}. For the off-diagonal part we have
		\begin{equation}\label{eq:nondiag}
 \sum_{0 \le i<j \le H-1} \frac{1}{X} \sum_{n=1}^{X} \alpha_y(n+i) \alpha_y(n+j) = \sum_{1 \le k \le H-1}  \frac{H-k}{X}\sum_{n=1}^{X} \alpha_y(n) \alpha_y(n+k) + E ,
		\end{equation}
	where 
	\begin{align*}
	E &= \frac{1}{X}\sum_{1 \le k \le H-1} \left(\sum_{n=1}^{H-k} \alpha_y(n)\alpha_y(n+k) (n-(H-k)) + \sum_{n=X+1}^{X+H-k} \alpha_y(n) \alpha_y(n+k) (X+H-k-n)\right)\\
	&=\frac{1}{X}\sum_{1 \le k \le H-1} \sum_{j=1}^{H-k-1}\left(\sum_{1 \le n \le j} \alpha_y(n+X)\alpha_y(n+X+k)-\sum_{1 \le n \le j}\alpha_y(n)\alpha_y(n+k)\right) \\
	& \ll \frac{H^2}{X} \left( H \Prob_y u^{-u(1+o(1))} + \Psiomega(y^u,y)\right)
\end{align*}
by partial summation and Lemma~\ref{lem:sievecomb} with $k=2$, where we have used the fact that $\prod_{p\le y}(1-\nu_p(0,k)/p) \le \Prob_y$.
The lemma is also used to estimate the main term in \eqref{eq:nondiag}, giving
		\begin{multline*}  \sum_{0 \le i<j \le H-1} \frac{1}{X} \sum_{n=1}^{X} \alpha_y(n+i) \alpha_y(n+j) =\sum_{\substack{1 \le k \le H-1\\ 2 \mid k}} (H-k) \prod_{p \le y} \left(1-\frac{\nu_p(0,k)}{p}\right) \\
			+ O\left(H^2\Prob_y u^{-u(1+o(1))} + \frac{H^2 \Psiomega(y^u,y)}{X}\right),
		\end{multline*}
		where we make use of the fact that $\prod_{p \le y}(1-\nu_p(0,k)/p)$ vanishes if $k$ is odd because of the $p=2$ factor. In summary,
		\[ V(X,H,y) = U+O\left( H^2\Prob_y  u^{-u(1+o(1))}+ \frac{H^2\Psiomega(y^u,y)}{X}\right)\] 
		for 
		\[ U = H\Prob_y - H^2\Prob^2_y +2\sum_{1 \le k < \frac{H}{2}} \left(\frac{H}{2}-k\right) \prod_{2<p \le y} \left(1-\frac{\nu_p(0,k)}{p}\right),\] 
		where we have replaced $k$ with $2k$.
		Next we estimate the sum appearing in $U$. Letting \[ C_y := \prod_{2<p \le y} \left(1-\frac{2}{p}\right)\]
		we have $\prod_{2<p\le y} (1-\nu_p(0,k)/p) = C_y f_y(k)$ for a multiplicative function $f_y$ defined by \[ f_y(k) = \prod_{\substack{2<p \le y\\  p \mid k}} \frac{1-\frac{1}{p}}{1-\frac{2}{p}}.\]
		In this notation,
		\begin{equation}\label{eq:usingf}
		\sum_{1 \le k < \frac{H}{2}} \left(\frac{H}{2}-k\right) \prod_{2<p\le y} \left(1-\frac{\nu_p(0,k)}{p}\right) = C_y \sum_{1 \le k < \frac{H}{2}} \left(\frac{H}{2}-k\right)f_y(k).
	\end{equation}
		The function $f_y$ is related to $g_y$ defined in \eqref{eq:gydef} via 
	\[		f_y(n) = \sum_{d \mid n} \frac{g_y(d)}{d}.\]
Hence the sum in the right-hand side of \eqref{eq:usingf} is
		\begin{equation}\label{eq:APFloor}
		\sum_{d<\frac{H}{2}} \frac{g_y(d)}{d}  \sum_{1 \le i <\frac{H}{2d}} \left(\frac{H}{2}-di\right)= \sum_{d} g_y(d) \left( \frac{H}{2d}\left\lfloor \frac{H}{2d} \right\rfloor -\frac{1}{2} \left\lfloor \frac{H}{2d}\right\rfloor \left(\left\lfloor \frac{H}{2d} \right\rfloor +1\right)\right). 
	\end{equation}
	For any $x>0$,
\begin{equation}\label{eq:fraciden}
	x\left\lfloor x\right\rfloor -\frac{1}{2} \left\lfloor x\right\rfloor \left(\left\lfloor x\right\rfloor +1\right) = \frac{x^2}{2} - \frac{x}{2} + \frac{1}{2}\left\{ x\right\}\left(1-\left\{ x\right\}\right).
\end{equation}
Applying \eqref{eq:fraciden} with $x=H/(2d)$ it follows that the right-hand side of \eqref{eq:APFloor} is
\begin{equation}\label{eq:physical}
\frac{1}{2}\sum_{d} g_y(d) \left\{ \frac{H}{2d}\right\}\left(1-\left\{ \frac{H}{2d}\right\}\right) + \frac{H^2}{8}\sum_{d} \frac{g_y(d)}{d^2} -  \frac{H}{4}\sum_{d} \frac{g_y(d)}{d}.
\end{equation}
After multiplying by $2C_y$, the last two $d$-sums in \eqref{eq:physical} cancel with $H^2 \Prob_y^2$ and $H\Prob_y$, and we find that 
		\begin{equation}\label{eq:asneeded}
		V(X,H,y) = C_y \sum_{d} g_y(d)\left\{ \frac{H}{2d}\right\}\left(1-\left\{ \frac{H}{2d}\right\}\right)+ O\left( H^2\Prob_y  u^{-u(1+o(1))}+\frac{H^2 \Psiomega(y^u,y)}{X}\right)
	\end{equation}
holds as needed. 
\end{proof}
\section{Main term estimates: large \texorpdfstring{$y$}{y}}\label{sec:largey}
\begin{lem}\label{lem:bnds}
	Given $ x\ge 2$, $y \ge 2$ we have
	\begin{align*}
		\sum_{r \ge 1}  g_y(r)\left\{ \frac{x}{r}\right\} \left(1-\left\{ \frac{x}{r}\right\}\right) &\le \sum_{r \ge 1} g_y(r) \min\left\{ \frac{x}{r}, 1\right\} ,\\
		\sum_{r \ge 1}  g_y(r)\left\{ \frac{x}{r}\right\} \left(1-\left\{ \frac{x}{r}\right\}\right) &\ge x\sum_{r > x}  \frac{g_y(r)}{r}\left(1-\frac{x}{r}\right).
	\end{align*}
	In particular,
	\begin{equation}\label{eq:gytail}
		\sum_{r \ge 1}  g_y(r)\left\{ \frac{x}{r}\right\} \left(1-\left\{ \frac{x}{r}\right\}\right) = x\sum_{r > x} \frac{g_y(r)}{r} + O\left( \sum_{r \le x} g_y(r) + x^2\sum_{r > x} \frac{g_y(r)}{r^2} \right).
	\end{equation}
\end{lem}
\begin{proof}
	This follows from $\{t\}(1-\{t\}) \le \min\{t,1\}$ for $t \ge 0$ and $\{t\}(1-\{t\})\ge t(1-t)$ for $0 \le t < 1$.
\end{proof}
\begin{lem}\label{lem:tailbnds}
Let $x \ge 2$, $y \ge 2$. Set $u=\log x/ \log y$. Fix $\varepsilon>0$. If $u \le \exp((\log y)^{3/5-\varepsilon})$, then
	\[ \sum_{r \le x}  g_y(r) \ll_{\varepsilon} x \rho(u) \qquad \text{and} \qquad  \sum_{r > x}  \frac{g_y(r)}{r^2} \ll_{\varepsilon} \frac{\rho(u)}{x} \]
	where $\rho$ is the Dickman function.
\end{lem}
\begin{proof}
	Let us define a multiplicative function $g$, supported on odd squarefrees, by $g(p)=p/(p-2)$ for all primes greater than $2$. In this way, $0 \le g_y \le g$. We have  $ \sum_{r \le x} g_y(r) \le \sum_{r \le x} g(r) \ll x$ by \cite[Cor.~2.5]{Montgomery2007}. In particular,
	\[ \sum_{r > x} \frac{g_y(r)}{r^2} \le \sum_{r > x} \frac{g(r)}{r^2} \ll \sum_{2^k>x/2} \frac{\sum_{r<2^k}g(r)}{4^k} \ll \frac{1}{x}.\]
	This proves the result as long as $y \ge x$, since $\rho(u)=1$ when $u \in [0,1]$. From now on we assume $1 \le u \le \exp((\log y)^{3/5-\varepsilon})$. In this case, the required bound for $\sum_{r \le x} g_y(r)$ is a special case of Corollaire 2.3 of Tenenbaum and Wu \cite{Tenenbaum2003}, and it remains to bound $\sum_{r>x} g_y(r)/r^2$. We again use dyadic dissection. Let $k_0$ be the largest $k$ satisfying
	\[\frac{\log(2^k)}{\log y} \le \exp((\log y)^{3/5-\varepsilon/2}).\]
	By the bound for $\sum_{r \le x} g_y(r)$ (with $\varepsilon/2$ in place of $\varepsilon$),
	\begin{equation}\label{eq:dyd} \sum_{r > x} \frac{g_y(r)}{r^2} \ll \sum_{2^k>x/2} \frac{\sum_{r<2^k}g_y(r)}{4^k} \ll \sum_{\substack{2^k>x/2\\ k < k_0}} \frac{2^k \rho\left( \frac{\log (2^k)}{\log y}\right)}{4^k} + \sum_{ k \ge k_0} \frac{(\log y)^2 \Psi(2^k,y)}{4^k},
	\end{equation}
	where we have used the pointwise bound $g_y \ll (\log y)^2$.
	As $\rho$ is nonincreasing, the first sum in the right-hand side of \eqref{eq:dyd} is
	\[ \ll \frac{\rho\left( u- \frac{\log 2}{\log y}\right)}{x} \ll \frac{\rho(u)}{x},\]
	where the last inequality follows from $\rho(u-t) \ll \rho(u)\exp(O(t \log u))$ for $t \in [0,1]$ \cite[Lem.~1]{Hildebrand1984}. We have $\Psi(2x,y)\le 2\Psi(x,y)$ for $x>y\ge C$ \cite{Hildebrand1985}, and $\Psi(x,y) \ll_{\varepsilon} x\rho(u)$ for $u \le \exp((\log y)^{3/5-\varepsilon})$ by \eqref{eq:psirho}. It follows that the second sum in the right-hand side of \eqref{eq:dyd} is
	\[ \ll \frac{(\log y)^2 \Psi(2^{k_0},y)}{4^{k_0}} \ll_{\varepsilon} \frac{(\log y)^2\rho\left( \frac{\log(2^{k_0})}{\log y}\right)}{2^{k_0}} \ll \frac{(\log y)^2 \rho(u)}{2^{k_0}} \ll_{\varepsilon} \frac{\rho(u)}{x}\]
	as needed.
\end{proof}
\begin{lem}\label{lem:tailasymptotics}
	Let $x \ge 2$, $y\ge 2$. Set $u=\log x / \log y$. If $y\ge \exp(C(\log x)^{1/2}(\log \log x)^{1/2})$, we have
	\begin{equation}\label{eq:ygreater}
		\sum_{r> x} \frac{g_y(r)}{r} = \prod_{2 <p \le y} \left(1 +\frac{1}{p-2}\right) \lambda(u)\left(1+O\left( \frac{(u+1)^2\log(u+2)}{\log x}\right) \right)
	\end{equation}
where $\lambda$ was defined in \eqref{eq:deflam}. 
\end{lem}
\begin{proof}
	Let us first consider $x \ge y \ge \exp(C(\log x)^{1/2}(\log \log x)^{1/2})$. Let $g$ be the multiplicative function supported on odd squarefrees and defined by $g(p)=p/(p-2)$ for all primes greater than $2$. Then the sum in question coincides with $\sum_{r>x,\, r\text{ is }y\text{-smooth}}g(r)/r$, a sum which was studied (for general $f$ in place of $g$) in \cite{Tenenbaum2008}. In the considered range, \eqref{eq:ygreater} is a special case of Th\'eor\`eme 3.2 of Tenenbaum and Wu \cite{Tenenbaum2008}. It remains to consider $y > x$. In this case we  may write the sum as
	\begin{align*}
		\sum_{r >x } \frac{g_y(r)}{r} &=\sum_{r \ge 1} \frac{g_y(r)}{r} - \sum_{r \le x} \frac{g_y(r)}{r} = \prod_{2<p \le y} \left(1+\frac{1}{p-2}\right)- \sum_{r\le x} \frac{g_x(r)}{r}\\
		&=\prod_{2<p \le y} \left(1+\frac{1}{p-2}\right)-\prod_{2<p \le x} \left(1+\frac{1}{p-2}\right)+\sum_{r >x } \frac{g_x(r)}{r}.
	\end{align*}
	Applying \eqref{eq:ygreater} with $y=x$ we find
	\[ \sum_{r>x} \frac{g_x(r)}{r} = \prod_{2 < p \le x} \left(1+\frac{1}{p-2}\right)\lambda(1) \left(1+O\left(\frac{1}{\log x}\right)\right).\]
	The result now follows by using $\prod_{x < p \le y} (1+1/(p-2)) = u^{-1}(1+O(1/\log x))$ (consequence of \eqref{eq:mertens3}) and recalling \eqref{eq:lam01}.
\end{proof}
The range of Lemma \ref{lem:tailasymptotics} can be improved to $y \ge \exp((\log \log x)^{5/3+\varepsilon})$ by invoking Th\'eor\`eme 3.3 in Tenenbaum and Wu \cite{Tenenbaum2008}.
\section{Main term estimates: small \texorpdfstring{$y$}{y}}\label{sec:saddle}
\subsection{Hall's identities}
We write $s \in \CC$ as $s=\sigma + it$. 
\begin{lem}\cite[Cors.~10.3--10.4]{Montgomery2007}\label{lem:functional}
	We have, for all $s \neq 1$, $\zeta(s) = \zeta(1-s)\chi(s)$ where 
	\[ \chi(s):=\pi^{s-\frac{1}{2}}\frac{\Gamma\left(\frac{1-s}{2}\right)}{\Gamma\left(\frac{s}{2}\right)}=\frac{(2\pi)^s}{\pi}\Gamma(1-s)\sin \frac{\pi s}{2}.\]
\end{lem}
\begin{lem}\cite[Eq.~(4.12.3)]{Titchmarsh}\cite[Cor.~10.5]{Montgomery2007}\label{lem:chisize}
	Fix $A>0$. For $|\sigma| \le A$ and $|t| \ge 1$ we have
	\[ \chi(s)= \left(\frac{2\pi}{|t|}\right)^{s-\frac{1}{2}}e^{i\left(t\pm\frac{\pi}{4}\right)}\left(1+O_A\left(|t|^{-1}\right)\right), \]
	where the sign in the exponent of $e$ is `$+$' if $t>0$, and is `$-$' if $t<0$. In particular $|\chi(s)| \asymp_A (|t|+4)^{\frac{1}{2}-\sigma}$.
\end{lem}
\begin{lem}\cite[Cor.~1.17]{Montgomery2007}\label{lem:convex}
	We have $\zeta(s) \ll (1+|t|^{1-\sigma})\min\{1/|\sigma-1|,\log |t|\}$ for $|t|\ge 2$ and $1/2 \le \sigma \le 2$. We have $\zeta(s) \ll  1/|\sigma-1|$ if $|t|\le 2$ and $1/2 \le \sigma \le 2$.
\end{lem}
Lemma~\ref{lem:perron rep} and Proposition~\ref{prop:tailI} below are implicit in Hall's work on squarefrees \cite[pp.~9--12]{Hall1982}. 
\begin{lem}[Hall]\label{lem:perron rep}
	Let $x>0$. For any $r \in (1/2,1)$,
	\begin{equation}\label{eq:abs}  \frac{1}{2\pi i} \int_{(r)} \frac{ \zeta(s-1)x^s}{s(s-1)}{\rm d}s = \frac{1}{2}\{ x\} (1-\{x\})
	\end{equation}
	and the integral converges absolutely. For any $r \in (0,1/2]$, \eqref{eq:abs} continues to hold if the integral is interpreted in principal value sense.
\end{lem}
\begin{proof}
	By Perron's formula, we have for all $r>1$
	\[ \frac{1}{2\pi i} \int_{(r)}\zeta(s)\frac{x^s}{s}{\rm d}s = \begin{cases} \sum_{n \le x} 1 & \text{if }x \notin \NN,\\ \sum_{n \le x-1} 1 + \frac{1}{2} &\text{if }x \in \NN\end{cases} = \lfloor x\rfloor - \mathbf{1}_{x \in \NN}\frac{1}{2}.  \]
	Similarly, for all $r>2$,
	\[ \frac{1}{2\pi i} \int_{(r)}\zeta(s-1)\frac{x^s}{s}{\rm d}s = \begin{cases} \sum_{n \le x} n & \text{if }x \notin \NN,\\ \sum_{n \le x-1} n + \frac{x}{2} &\text{if }x \in \NN\end{cases} = \sum_{n \le x} n - \mathbf{1}_{x \in \NN}  \frac{x}{2}.  \]
	Changing $s$ to $s-1$ in the first integral and taking a linear combination of these two integrals, we find that for all $r >2$ we have
	\begin{equation}\label{eq:zetassminus}
		\begin{split} \frac{1}{2\pi i} \int_{(r)}\frac{ \zeta(s-1)x^s}{s(s-1)}{\rm d}s &= \frac{1}{2\pi i} \int_{(r)}\frac{ \zeta(s-1) x^s}{s-1}{\rm d}s - \frac{1}{2\pi i} \int_{(r)}\frac{ \zeta(s-1) x	^s}{s}{\rm d}s\\
			&= x \left( \lfloor x\rfloor - \mathbf{1}_{x \in \NN}\frac{1}{2} \right) - \bigg( \sum_{n \le x} n - \mathbf{1}_{x \in \NN}  \frac{x}{2} \bigg)=  x\lfloor x \rfloor -\frac{1}{2}\lfloor x \rfloor (\lfloor x \rfloor +1).
		\end{split}
	\end{equation}
	We now manipulate the left-hand side of \eqref{eq:zetassminus}.
	In $\Re s > 0$, the integrand $ \zeta(s-1) x^s/(s(s-1)) $ has a simple pole at $s=1$ with residue $\zeta(0)x = -x/2$ and a simple pole at $s=2$ with residue $x^2/2$. Moreover, the estimates for $\zeta$ given in Lemmas~\ref{lem:functional}--\ref{lem:convex} and the presence of $s(s-1)$ in the integrand allow us to shift the contour of integration to the left of the pole at $s=2$, obtaining from \eqref{eq:zetassminus}
	\[ \frac{1}{2\pi i} \int_{(r)} \frac{ \zeta(s-1)x^s}{s(s-1)}{\rm d}s =x\lfloor x \rfloor -\frac{1}{2}\lfloor x \rfloor (\lfloor x \rfloor +1) -\frac{x^2}{2} \]
	for all $r \in (1,2)$, and shift further more to the left of the pole at $s=1$, obtaining, by \eqref{eq:fraciden},
	\[ \frac{1}{2\pi i} \int_{(r)} \frac{ \zeta(s-1)x^s}{s(s-1)}{\rm d}s = x\lfloor x \rfloor -\frac{1}{2}\lfloor x \rfloor (\lfloor x \rfloor +1) -\frac{x^2}{2} + \frac{x}{2}=\frac{1}{2}\{x\}(1-\{x\}) \]
	for all $r \in (1/2,1)$, as the integral converges absolutely for $r$ in that range, since $|\zeta(s-1)/(s(s-1))| \ll_r (|t|+4)^{-1/2-\sigma}$ for $|t| \ge 1$ and $\Re s =r$ by Lemmas~\ref{lem:functional}--\ref{lem:convex}. This concludes the first part of the lemma.
	
	We now consider the second part. Let $r \in (0,1/2]$. Let $T\ge 2$. We apply Cauchy's integral theorem to the rectangle with vertices $r-iT$, $r+iT$, $3/4+iT$, $3/4-iT$. Then
	\[\int_{r-iT}^{r+iT} \frac{ \zeta(s-1)x^s}{s(s-1)}{\rm d}s = \int_{3/4-iT}^{3/4+iT} \frac{ \zeta(s-1)x^s}{s(s-1)}{\rm d}s + \int_{r-iT}^{3/4-iT} \frac{ \zeta(s-1)x^s}{s(s-1)}{\rm d}s +  \int_{3/4+iT}^{r+iT} \frac{ \zeta(s-1)x^s}{s(s-1)}{\rm d}s.\]
	Letting $T\to \infty$ and invoking the first part with $r=3/4$, it suffices to show that the integrals
	\[  \int_{r-iT}^{3/4-iT} \frac{ \zeta(s-1)x^s}{s(s-1)}{\rm d}s\qquad \text{and}\qquad \int_{3/4+iT}^{r+iT} \frac{ \zeta(s-1)x^s}{s(s-1)}{\rm d}s\]
	go to zero with $T$. Both integrals are 
	\[ \ll \frac{1}{T^2} \int_{r}^{3/4} x^{\sigma} |\zeta(\sigma-1\pm iT)| {\rm d}\sigma \ll \frac{1}{T^2} \int_{r}^{3/4} x^{\sigma} T^{\frac{3}{2}-\sigma}  {\rm d}\sigma\]
	by Lemmas~\ref{lem:functional}-\ref{lem:chisize}. If $T \gg x^2$, this is $\ll x^r/T^{r+1/2}$, which goes to $0$ as $T \to \infty$.
\end{proof}
The following is a truncated version of Lemma~\ref{lem:perron rep}.
\begin{proposition}[Hall]\label{prop:tailI}
	Uniformly for $x>0$, $r \in (0,1)$ and $T \ge 1$ we have
	\begin{equation}\label{eq:truncatedhall}
		\int_{r+iT}^{r+i\infty} \frac{\zeta(s-1)x^s}{s(s-1)}{\rm d}s  \ll \frac{x^r}{T^{\frac{1}{2}+r}(1-r)} +a_r(x,T)
	\end{equation}
	where
	\begin{equation}\label{eq:ardef}
		a_r(x,T) := (x/T)^{1/2+r} \sum_{n \le T/x} n^{r-3/2} + \min\{1,x/T\} \ll \min\{1,(x/T)^r\}.
	\end{equation}
	An identical bound holds if we integrate from $r-i\infty$ to $r-iT$.
\end{proposition}
\begin{proof}
	We apply the functional equation, stated in Lemma~\ref{lem:functional}, to the integrand in \eqref{eq:truncatedhall}, which replaces the integral in question with 
	\begin{equation}\label{eq:considered}
		\int_{r+iT}^{r+i\infty} \frac{\zeta(2-s)\chi(s-1)x^s}{s(s-1)}{\rm d}s.
	\end{equation}
	By Lemma~\ref{lem:chisize},
	\begin{equation}\label{eq:chisizeapp} \frac{\chi(s-1)}{s(s-1)}=-t^{-2} \left(\frac{2\pi}{t}\right)^{s-\frac{3}{2}}e^{i\left(t+\frac{\pi}{4}\right)}\left(1+O\left(t^{-1}\right)\right)
	\end{equation}
	uniformly for $t \ge 1$ and $\sigma \in [0,1]$. The contribution of the error term in \eqref{eq:chisizeapp} to \eqref{eq:considered} is, by the triangle inequality, at most
	\[ \ll x^{r} \int_{T}^{\infty} |\zeta(2-r-it)| t^{-\frac{3}{2}-r}{\rm d}t \ll x^r \zeta(2-r) T^{-\frac{1}{2}-r},\]
	which is acceptable. The contribution of the main term in \eqref{eq:chisizeapp} to \eqref{eq:considered} can be written as $-e^{i\pi/4}(2\pi)^{r-3/2}i$ (a constant) times
	\begin{equation}\label{eq:MTint}
		x^r \int_{T}^{\infty}
		\zeta(2-r-it)   \left(\frac{2\pi ex}{t}\right)^{it} t^{-\frac{1}{2}-r} {\rm d}t .
	\end{equation}
	The rest of the proof bounds \eqref{eq:MTint}. We introduce
	\[ J(t):=\int_{T}^{t} \zeta(2-r-iu)\left( \frac{2\pi ex}{u}\right)^{iu} {\rm d}u.\]
	We shall estimate $J$. Using the definition of $\zeta$ and interchanging order of summation, we see
	\[ J(t) = \sum_{n \ge 1} n^{r-2} J_{nx}(t), \qquad J_{nx}(t):= \int_{T}^{t}\left( \frac{2\pi enx}{u}\right)^{iu}{\rm d}u=\int_{T}^{t}e^{iF_{nx}(u)}{\rm d}u\]
	for
	\[ F_{a}(u):= u \log \left(\frac{2\pi ea}{u}\right), \qquad F'_a(u) = \log\left( \frac{2\pi a}{u}\right).\]
	The theory of oscillatory integrals easily bounds $J_{nx}$. Indeed, $F_a'$ is monotone decreasing in $(0,\infty)$. In $(0,a]$, $F'_a(u)\ge 1$, while in $[20a,\infty)$ we have $F'_a(u)\le -1$. Thus, Lemma 4.2 of \cite{Titchmarsh} implies that  $\int_{I} e^{iF_{a}(u)} {\rm d}u \ll1$ 	for any interval $I$ that satisfies $I \subseteq [1,a]\cup [20a,\infty)$. It follows that \[J_{nx}(t) \ll 1 + \left|\int_{I_{nx}} e^{iF_{nx}(u)}{\rm d}u\right|\]
	for the (possibly empty) interval $I_{nx} = [T,t]\cap [nx,20nx]$. If $nx \ge t$, then $I_{nx}= \varnothing$ and so $J_{nx}(t) \ll 1$. Otherwise, since $F''_a(u) = -1/u\ll -1/a$ if $u \in [a,20a]$, Lemma 4.4 of \cite{Titchmarsh} implies that 
	\[	\int_{I_{nx}} e^{iF_{nx}(u)}{\rm d}u \ll \sqrt{nx}.\]
	In summary, for all $n$, $x$ and $t$,
	\begin{equation}\label{eq:unif}
		J_{nx}(t) \ll 1 + \sqrt{nx} \mathbf{1}_{t>nx}.
	\end{equation}
	Integration by parts allows us to express \eqref{eq:MTint} as
	\begin{equation}\label{eq:ibp} \left(\frac{1}{2}+r\right)x^r \int_{T}^{\infty} J(t) t^{-\frac{3}{2}-r} {\rm d}t = \left(\frac{1}{2}+r\right)x^r\sum_{n\ge 1} n^{r-2} \int_{T}^{\infty} J_{nx}(t) t^{-\frac{3}{2}-r}{\rm d}t . 
	\end{equation}
	To conclude we input the estimate \eqref{eq:unif} into \eqref{eq:ibp} and simplify.
\end{proof}
For $\Re s> 0$ let 
\begin{equation}\label{eq:Gys}
	G_y(s) : =\sum_{n \ge 1} \frac{g_y(n)}{n^s} = \prod_{2<p\le y}\left(1+\frac{1}{(p-2)p^{s-1}}\right).
\end{equation}
The sum and product in \eqref{eq:Gys} are finite since $g_y$ has finite support. 
We consider the integral
\begin{equation}\label{eq:Ic} I_r(x,y):=\frac{1}{2\pi i}\int_{(r)} x^{s}  G_y(s)\frac{\zeta(s-1){\rm d}s}{s(s-1)}.
\end{equation}
By Lemma~\ref{lem:perron rep}, $I_r$ converges absolutely for $r \in (1/2,1)$, and in principal value sense for $r \in (0,1)$.
\begin{cor}\label{cor:ic}
	Given $r \in (0,1)$, $y \ge 2$, $x \ge 1$ and $T \ge 1$ we have, for $I_r$ defined in \eqref{eq:Ic},
	\begin{align*}
		&\sum_{n \ge 1}  g_y(n)\left\{ \frac{x}{n}\right\} \left(1-\left\{ \frac{x}{n}\right\}\right) \\
		&= 2I_r(x,y)= \frac{2}{2\pi i}
		\int_{r-iT}^{r+iT} x^s G_y(s)\frac{\zeta(s-1)}{s(s-1)}{\rm d}s + O\bigg( \frac{x^rG_y(r) }{T^{\frac{1}{2}+r}(1-r)}+ \sum_{n \ge 1} g_y(n) a_{r}(x/n,T)\bigg)
	\end{align*}
	where $a_r$ is defined in \eqref{eq:ardef}.
\end{cor}
\begin{proof}
	For each $n$ in the (finite) support of $g_y$ we apply Lemma~\ref{lem:perron rep} and Proposition~\ref{prop:tailI} with $x/n$ in place of $x$, multiply by $g_y(n)$, and sum over $n$.
\end{proof}

\begin{remark}
	In the course of establishing Proposition~\ref{prop:sieveapproach}, we derived  \eqref{eq:asneeded} from \eqref{eq:usingf} and \eqref{eq:APFloor} using an elementary argument. We demonstrate how one can deduce \eqref{eq:asneeded} from \eqref{eq:usingf} via complex analysis, in the spirit of \cite{Montgomery2002}. By \eqref{eq:usingf}, \eqref{eq:APFloor} and \eqref{eq:zetassminus}, it follows that
	\[  2\sum_{1 \le k < \frac{H}{2}} \left(\frac{H}{2}-k\right) \prod_{2<p\le y} \left(1-\frac{\nu_p(0,k)}{p}\right)= \frac{2C_y}{2\pi i}\int_{(r)} (H/2)^{s} \zeta(s-1) G_y(s)\frac{{\rm d}s}{s(s-1)} \]
	for every $r>2$, where $G_y$ is defined in \eqref{eq:Gys}. We shift the contour to the left, namely to $(r')$ for some $r'\in (1/2,1)$. The shift to $(r')$, as well as convergence on the line $(r')$ is justified by the bounds in Lemmas~\ref{lem:functional}--\ref{lem:convex}. By Cauchy's residue theorem, we pick up two residues due to two simple poles: one is at $s=2$, due to $\zeta(s-1)$, where the residue is $H^2 G_y(2)/8$, which, after multiplying by $2C_y$, cancels with $-H^2\Prob^2_y$. Another pole is at $s=1$, due to $1/(s-1)$, where the residue is $H\zeta(0) G_y(1)/2$ which, after multiplying by $2C_y$, cancels  with $H\Prob_y$. We find that
	\begin{equation}\label{eq:asneeded2}
		V(X,H,y) = \frac{C_y}{\pi i}\int_{(r')} (H/2)^{s} \zeta(s-1) G_y(s)\frac{{\rm d}s}{s(s-1)} + O\left( H^2\Prob_y  u^{-u(1+o(1))}+\frac{H^2 \Psiomega(y^u,y)}{X}\right)
	\end{equation}
	for every $r' \in (1/2,1)$. The main term in \eqref{eq:asneeded2} coincides with $M(H,y)$ by the first equality in Corollary~\ref{cor:ic}, giving \eqref{eq:asneeded}. This gives a complex analytic proof of Proposition~\ref{prop:sieveapproach}. We used poles of the integrand $(H/2)^s \zeta(s-1)G_y(s)/(s(s-1))$ at $s=1$ and $s=2$ to cancel the contribution of the diagonal terms $H\Prob_y - H^2 \Prob_y^2$. A similar phenomenon occurs when studying the variance of primes, see the very last paragraph of \cite{Montgomery2002}, where the authors say ``it seems fortuitous that $T(1)=U(0)=2/c$ and that $U'(0)=0$. But miracles do not happen by accident, so it seems that there is something going here that remains to be understood''. We see that this phenomenon repeats when computing the variance of $\alpha_y$.
\end{remark}
\subsection{Saddle point analysis}
For $\Re s >0$ let $\zeta(s,y)=\prod_{p \le y}(1-p^{-s})^{-1}$ be the partial zeta function and
\[ \zeta_1(s,y) := \frac{\zeta(s,y)}{\zeta(2s,y)} = \prod_{p \le y} (1+p^{-s}) = \sum_{n\text{ is }y\text{-smooth}}\mu^2(n)n^{-s}.\]	  
Let $\alpha=\alpha(x,y)$ be the unique  solution to the following equation in $s>0$:
\[\log x = \sum_{p \le y} \frac{\log p}{p^{s}+1} = -\frac{\partial}{\partial s}\log \zeta_1(s,y).\]
We use the notation $\alpha$ throughout. Since $\sum_{p \le y} \log p/(p^s+1)$ is decreasing in $s$, we see that $\alpha$ exists as long as $\vartheta(y) > 2 \log x$, where
\[\vartheta(y):=\sum_{p \le y} \log p \sim y\]
is the first Chebyshev function. In this subsection we make use of the results in La Bret\`eche and Tenenbaum \cite{LaB} where $\Psi_{\mu^2}$ is studied via the saddle point method applied to $\zeta_1$.
\begin{lem}\cite[Lem.~1]{Hildebrand1984}\label{lem:hildlem1}
	Given $u \ge 1$ let $\xi=\xi(u) \ge 0$ be the solution to $e^{\xi}=1+u\xi$. We have $\xi(u) \sim \log u$ as $u \to \infty$.
	\end{lem}
\begin{lem}\cite[Lem.~2.8]{LaB}\label{lem:hild2}
	Fix $\varepsilon>0$. If $x \ge y \ge(\log x)^{1+\varepsilon}$ and $x\ge C_{\varepsilon}$, then (in the notation of Lemma~\ref{lem:hildlem1}) \[ \alpha = 1- \frac{\xi(u)}{\log y}+ O_{\varepsilon}\left(\frac{1}{\log x \log y} + e^{-\sqrt{\log y}}\right).\]
	If $(\log x)^3 \ge \vartheta(y)>2 \log x$, then
	\[ \alpha = \frac{1+O(1/\log y)}{\log y}\log \left( \frac{\vartheta(y)}{\log x}-1\right).\]
\end{lem}
\begin{lem}\cite[Lem.~2.10]{LaB}\label{lem:hild} Suppose $x \ge y \ge 2$ and $\vartheta(y)>2\log x$. If $|t|\le 1/\log y$, then
\[ \left| \frac{\zeta_1(\alpha+it,y)}{\zeta_1(\alpha,y)}\right| \le \exp\left(-ct^2 \log x \log y \right) \]
	for some absolute $c>0$. For any $\varepsilon>0$ we have
	\begin{equation}\label{eq:uratio}
	\left| \frac{\zeta_1(\alpha+it,y)}{\zeta_1(\alpha,y)}\right| \le \exp\left( -\frac{cut^2}{(1-\alpha)^2+t^2}\right)
		\end{equation}
	for $x \ge C_{\varepsilon}$, $1/\log y < |t| \le \exp((\log y)^{3/2-\varepsilon})$. Here $u=\log x / \log y$.
\end{lem}
Recall $G_y$ was defined in \eqref{eq:Gys}. We define $H_y$ by 
\begin{equation} \label{eq:defh}
	G_y(s) = H_y(s) \zeta_1(s,y).
\end{equation}
The Euler product for $H_y$ is
\begin{equation}\label{eq:HyEuler}
		H_y(s) = (1+2^{-s})^{-1} \prod_{2<p \le y} \left( 1+\frac{2}{(p^{s}+1)(p-2)}\right).
\end{equation}
\begin{lem}\label{lem:hy}
	Fix $\varepsilon>0$. Suppose $x \ge y\ge (2+\varepsilon)\log x$ and $x \ge C_{\varepsilon}$. We have $1\ll H_y(\alpha) \ll \log y$ and $H_y(\alpha+it) \ll_{\varepsilon} \exp(C_{\varepsilon}\log y/\log \log y)$ for all $t \in \RR$. If $y \ge (\log x)^{1+\varepsilon}$, then $H_y(\alpha+it) \ll_{\varepsilon} 1$ for all $t \in \RR$.
\end{lem}
\begin{proof}
	By Mertens' theorem, $H_y(0) \ll \log y$, and $H_y(1)\gg 1$. This gives the bounds on $H_y(\alpha)$. For the  bounds on $H_y(\alpha+it)$ we first observe
	\begin{equation}\label{eq:plower}
	|p^{\alpha+it}+1| \ge p^{\alpha}-1 \ge \alpha \log p.
\end{equation}
If $y \ge (2+\varepsilon)\log x$, then $|p^{\alpha+it}+1| \gg_{\varepsilon} \log p /\log y$ by \eqref{eq:plower} and Lemma~\ref{lem:hild2}. Hence
	\[ H_y(\alpha+it) \ll_{\varepsilon} \log y \prod_{2 <p \le y} \left(1 + \frac{C_{\varepsilon}\log y}{p\log p}\right)\]
	which is  $\ll_{\varepsilon} \exp(C_{\varepsilon}\log y/\log \log y)$ as needed. If $y \ge (\log x)^{1+\varepsilon}$, then $\alpha \gg_{\varepsilon} 1$ by Lemma~\ref{lem:hild2} and so 
	\[ H_y(\alpha+it) \ll_{\varepsilon} \prod_{2 <p \le y} (1+C_{\varepsilon}/(p\log p)) \ll_{\varepsilon} 1 \]
	by \eqref{eq:plower} and Chebyshev's bound $\vartheta(t) =O( t)$.
\end{proof}
\begin{lem}\label{lem:taillemma}
Fix $\varepsilon>0$. Suppose $(\log x)^{3/2} \ge y \ge (2+\varepsilon)\log x$ and $x \ge C_{\varepsilon}$. For $1 \le T \le \exp(c\sqrt{\log x \log y})$,
\[ \sum_{n\ge 1} g_y(n) \min\{(x/(nT))^{1/2+\alpha},1\} \ll (\log y)^2 \frac{\Psi_{\mu^2}(x,y)}{T^{\alpha}}.\]
\end{lem}
\begin{proof}
Since $g_y(n) \ll (\log y)^2 \mu^2(n)\mathbf{1}_{n\text{ is }y\text{-smooth}}$, splitting the range of $n$ into dyadic intervals shows
\[ \sum_{n\ge 1} g_y(n) \min\{(x/(nT))^{1/2+\alpha},1\} \ll (\log y)^2 \sum_{k\ge 0} \frac{\Psi_{\mu^2}(2^{k} x/T,y)}{2^{k(1/2+\alpha)}}.\]
It is known that $\Psi_{\mu^2}(2x_0,y) \le 2^{\alpha(x_0,y)}(1+C\sqrt{\log y/\log x_0}) \Psi_{\mu^2}( x_0,y)$ holds for $\vartheta(y)>2\log (2 x_0)$ \cite[Cor.~2.4]{LaB} and $\Psi_{\mu^2}(2x_0,y)\le 2^{\pi(y)}$ otherwise. These inequalities imply
\[ \sum_{k\ge 0} \frac{\Psi_{\mu^2}(2^{k} x/T,y)}{2^{k(1/2+\alpha)}} \ll  \Psi_{\mu^2}(x/T,y) + \frac{2^{\pi(y)}}{(e^{\vartheta(y)/2}T/x)^{1/2+\alpha}} \le \Psi_{\mu^2}(x/T,y)+1 \le 2\Psi_{\mu^2}(x/T,y) .\]
Iterating the inequality $\Psi_{\mu^2}(x_0,y)\le \Psi_{\mu^2}( ax_0,y)(1+C\sqrt{\log y/\log x_0})/a^{\alpha(x_0,y)} $ which holds for $\vartheta(y)>2\log (a x_0)$ and $1 \le a \le y$ \cite[Cor.~2.4]{LaB}, we find $\Psi_{\mu^2}(x/T,y) \ll \Psi_{\mu^2}(x,y)/T^{\alpha}$.	
\end{proof}
\begin{proposition}\label{prop:saddle prepare}
	Fix $\varepsilon>0$. Suppose $x \ge y \ge (2+\varepsilon)\log x$ and $x \ge C_{\varepsilon}$. Suppose further that $u \ge (\log \log (2y))^5$ where $u=\log x/\log y$. Then, for $T=\exp(\min\{cu^{1/3},(\log y)^{3/2-\varepsilon}\})$,
\[		I_{\alpha}(x,y) = \frac{1}{2\pi i}\int_{\alpha-\frac{i}{u^{1/3}\log y}}^{\alpha+\frac{i}{u^{1/3}\log y}}x^{s} G_y(s)\frac{\zeta(s-1){\rm d}s}{s(s-1)} + O_{\varepsilon}\left( \frac{x^{\alpha}\zeta_1(\alpha,y)H_y(\alpha)}{ T^{\frac{1}{4}}(1-\alpha)\alpha} + \frac{(\log y)^2\Psi_{\mu^2}(x,y)}{ T^{\alpha}}\right),\]
for $I_{\alpha}$ defined in \eqref{eq:Ic}.
\end{proposition}
\begin{proof}
Let $T$ be as in the statement of the proposition. By Corollary~\ref{cor:ic} with $r=\alpha$, 
	\begin{equation}\label{eq:ialphae1e2} I_{\alpha}(x,y) = \frac{1}{2\pi i}\int_{\alpha-iT}^{\alpha+iT} x^s G_y(s) \frac{\zeta(s-1){\rm d}s}{s(s-1)} + O(E_1 + E_2)
		\end{equation}
	for
	\[ 	E_1 = \frac{x^{\alpha}G_y(\alpha)}{T^{\frac{1}{2}+\alpha}(1-\alpha)}=\frac{x^{\alpha}\zeta_1(\alpha,y)H_y(\alpha)}{T^{\frac{1}{2}+\alpha}(1-\alpha)}  \qquad \text{and}\qquad E_2 = \sum_{n \ge 1} g_y(n) a_{\alpha}(x/n,T),\]
	where $a_{\alpha}$ is defined in \eqref{eq:ardef}. Clearly $E_1$ is acceptable. To bound $E_2$, we use the estimate $a_{\alpha}(x/n,T) \ll (x/(nT))^{\alpha}$ which implies
	\[ E_2 \ll (x/T)^{\alpha}G_y(\alpha) = \frac{x^{\alpha}\zeta_1(\alpha,y)H_y(\alpha)}{T^{\alpha}}.\]
	If $y > (\log x)^{3/2}$, then $\alpha \ge 1/4$ by Lemma~\ref{lem:hild2} and so $E_2$ is acceptable. If $y \le (\log x)^{3/2}$, then $\alpha \le 2/5$ and so $a_{\alpha}(x/n,T) \ll \min\{1,(x/(nT))^{1/2+\alpha}\}$. Thus
	\[ E_2 \ll \sum_{n \ge 1} g_y(n) \min\{(x/(nT))^{1/2+\alpha},1\},\]
	which is acceptable by Lemma~\ref{lem:taillemma}.
	Next we bound the contribution of $|\Im s| \in [1/\log y, T]$ to the integral in \eqref{eq:ialphae1e2}. By the functional equation as in Lemmas~\ref{lem:functional}--\ref{lem:chisize}, and  \eqref{eq:uratio},
	\begin{multline}\label{eq:above1log}
			\int_{\alpha+\frac{i}{\log y}}^{\alpha+iT}x^s G_y(s)\frac{\zeta(s-1)  {\rm d}s}{s(s-1)}\ll  x^{\alpha} \zeta_1(\alpha,y)\int_{\frac{1}{\log y}}^{T} \frac{(t+2)^{3/2-\alpha}\log (y(t+2))}{(t+\alpha)(t+1-\alpha)} \left| \frac{\zeta_1(\alpha+it,y)}{\zeta_1(\alpha,y)}\right| |H_y(\alpha+it)| {\rm d}t\\
			 \le \frac{x^{\alpha}\zeta_1(\alpha,y)}{(1-\alpha)\alpha}  \int_{\frac{1}{\log y}}^{T} (t+2)^{3/2-\alpha}\log (y(t+2)) \exp\left(  \frac{-c ut^2}{(1-\alpha)^2+t^2} \right) |H_y(\alpha+it)| {\rm d}t.
	\end{multline}
	Here the $\log(y(t+2))$ factor is a bound for $\zeta(2-s)$ (following from Lemma \ref{lem:convex}) which is needed if $\alpha$ is close to $1$. The contribution of $\max\{1-\alpha,1/\log y\} \le t \le T$ to the integral in the right-hand side of \eqref{eq:above1log} is
	\[ \ll \max_{t \in \RR}|H_y(\alpha+it)| \exp(-cu)\log y\int_{0}^{T} (t+2)^{2}{\rm d}t \ll \max_{t \in \RR}|H_y(\alpha+it)|\exp(-cu) T^{3}\]
	which is acceptable by Lemma~\ref{lem:hy}. The contribution of $1/\log y \le t \le \max\{1-\alpha,1/\log y\}$ to the integral in the right-hand side of \eqref{eq:above1log} is $0$ if $1/\log y \ge 1-\alpha$ and otherwise is (by Lemma~\ref{lem:hild2})
\[		\ll \log y \int_{\frac{1}{\log y}}^{1-\alpha}  \exp\left(\frac{-cu t^2}{(1-\alpha)^2} \right)|H_y(\alpha+it)| {\rm d}t \ll_{\varepsilon} \max_{|t| \le 1}|H_y(\alpha+it)|\exp\left( \frac{-c_{\varepsilon}u}{ \log^2 u}\right)\]
	which is also acceptable by Lemma~\ref{lem:hy}. Finally we turn to the contribution of $\int_{\alpha+i/(u^{1/3}\log y)}^{\alpha+i/\log y}$ to the integral in \eqref{eq:ialphae1e2}. Using the first part of Lemma~\ref{lem:hild}, it is at most $Cx^{\alpha}\zeta_1(\alpha,y)/((1-\alpha)\alpha)$ times
	\[\max_{|t|\le 1}|H_y(\alpha+it)|u^{1/3}\log y \int_{\frac{1}{u^{1/3}\log y}}^{\frac{1}{\log y}} \exp( -ct^2 \log x \log y){\rm d}t \ll \max_{|t|\le 1}|H_y(\alpha+it)|\exp(-cu^{1/3}), \]
	which is also acceptable. Here $O(u^{1/3}\log y)$ is a bound on $\zeta(2-s)$ where $\Re s = \alpha$ and $|\Im s| \ge 1/(u^{1/3}\log y)$.
\end{proof}
\begin{lem}\label{lem:h size}
Let $\beta \in (0,1]$ and $y\ge 2$.  We have $|(\log H_y(\beta+it))^{(j)}| \ll_{j} \beta^{-j}$ for $|t|\le c/\log y$ and $j \ge 1$. 
\end{lem}
\begin{proof}
When $|t| \le c/\log y$ and $p \le y$ we have
\begin{equation}\label{eq:low1}
|p^{\beta+it}+1| = |p^{\beta} (1+it\log p + O(t^2\log ^2p))+1| \gg p^{\beta}.
\end{equation}
Similarly, when $|t| \le c/\log y$ and $2<p \le y$ we have
\begin{equation}\label{eq:low2}
|p^{\beta+it}+1+2/(p-2)| \gg p^{\beta}.
\end{equation}
Let  $f_p(s) :=1+2/((p^s+1)(p-2))$ for each $p > 2$. By \eqref{eq:HyEuler}, $\log H_y(s)  =\sum_{2<p \le y} \log f_p(s) -\log(1+2^{-s})$. Induction on $j$ shows that
\[ (\log f_p(s))^{(j)} = \frac{(\log p)^j P_{2j-1}(p^s)}{(p-2)((p^s+1)(p^s+1+\frac{2}{p-2}))^j}\]
for $j \ge 1$, where $P_{2j-1}$ is a polynomial of degree $2j-1$ whose coefficients are $\ll_j 1$. Thus, from \eqref{eq:low1}--\eqref{eq:low2},
\[|(\log H_y(\beta+it))^{(j)}| \ll_j 1+\sum_{2<p\le y} \frac{(\log p )^j}{p^{1+\beta}} \ll_j \sum_{p\le y} \frac{(\log p )^j}{p^{1+\beta}} \ll_j \int_{2}^{y} \frac{(\log t)^{j-1}}{t^{1+\beta}}{\rm d}t.\]
The change of variables $t=e^{v/\beta}$ shows that the last integral is $\ll_j \beta^{-j}$.
\end{proof}

\begin{proposition}\label{prop:i is psi}
	Fix $\varepsilon>0$.	Suppose $x \ge y \ge (2+\varepsilon)\log x$, $u \ge (\log \log (2y))^5$ and $x \ge C_{\varepsilon}$ where $u=\log x/\log y $. Then
	\[ I_{\alpha}(x,y) = \frac{H_y(\alpha)\zeta(\alpha-1)}{\alpha-1} \Psi_{\mu^2}(x,y)(1+O_{\varepsilon}(u^{-1}+ \exp(-\alpha (\log y)^{3/2-\varepsilon})))\]
	for $I_{\alpha}(x,y)$ defined in \eqref{eq:Ic}.
\end{proposition}
\begin{proof}
	We consider $s=\alpha+it$ with $|t| \le 1/(u^{1/3}\log y)$. As in the proof of \cite[Lem.~11]{Hildebrand1986}, we have the following expansion:
	\begin{equation}\label{eq:taylorpartialzeta}
		\begin{split}
		\frac{\zeta_1(s,y) x^{s}}{s} &= \frac{\zeta_1(\alpha,y)x^{\alpha}}{s}e^{-\frac{\sigma_2 t^2}{2} -i\frac{\sigma_3t^3}{6}+ O(  \sigma_4t^4)}\\
		&=\frac{\zeta_1(\alpha,y)x^{\alpha}}{\alpha}e^{-\frac{\sigma_2 t^2}{2}} \left(1-\frac{it}{\alpha}-i\frac{\sigma_3t^3}{6}+O_{\varepsilon}\left( \sigma_3^2 t^6 +  \alpha^{-2}t^2+ \sigma_4 t^4\right)\right)
		\end{split}
	\end{equation}
	where $|\sigma_k| \ll_{\varepsilon} (\log y)^{k-1} \log x$  for $k =3,4$ and $\sigma_2\asymp_{\varepsilon} \log x \log y$ \cite[Lem.~2.9]{LaB}.
	We also have
	\[ \zeta(s-1)= \zeta(\alpha-1) (1 + a_1 t + a_2t^2 + a_3 t^3 + O(t^4))\]
	for bounded $a_i$ and 
	\[ H_y(s)= H_y(\alpha) (1 + b_1 t + b_2t^2 + b_3 t^3 + O_{\varepsilon}(b_4 t^4))\]
	for $b_i\ll_{i,\varepsilon} \alpha^{-i}$ by Lemma~\ref{lem:h size} with $\beta=\alpha$. For $1/(s-1)$, 
	\[ \frac{1}{s-1} = \frac{1}{\alpha-1} \left( 1- \frac{it}{\alpha-1}- \frac{t^2}{(\alpha-1)^2} +\frac{it^3}{(\alpha-1)^3} + O\left( \frac{t^4}{(\alpha-1)^4}\right)\right).\]
	Hence
	\[ \frac{\zeta(s-1) H_y(s)}{s-1} = \frac{\zeta(\alpha-1)H_y(\alpha)}{\alpha-1} \left( 1+c_1t+c_3t^3 + O_{\varepsilon}\left( \alpha^{-2}(1-\alpha)^{-2}t^2\right)\right)\]
	where $c_i \ll_{\varepsilon}\alpha^{-i}(1-\alpha)^{-i}$. Multiplying \eqref{eq:taylorpartialzeta} by the Taylor expansion for $(\zeta(s-1) H_y(s))/(s-1)$ and integrating over $|t| \le 1/(u^{1/3}\log y)$, we have 
	\[ \int_{-\frac{1}{u^{1/3}\log y}}^{\frac{1}{u^{1/3}\log y}} \frac{\zeta_1(\alpha+it,y)x^{\alpha+it}\zeta(\alpha-1+it)H_y(\alpha+it)}{(\alpha+it)(\alpha+it-1)}{\rm d}t \\= \frac{\zeta_1(\alpha,y)x^{\alpha}\zeta(\alpha-1)H_y(\alpha)}{\alpha(\alpha-1)} \int_{-\frac{1}{u^{1/3}\log y}}^{\frac{1}{u^{1/3}\log y}} e^{-\frac{\sigma_2 t^2}{2}}F(t){\rm d}t\]
	where
	\[ F =1+ d_1 t+ d_2t^3 + O_{\varepsilon}\left(\left(\sigma_3^2+\frac{\sigma_3}{\alpha^3(1-\alpha)^3}\right)t^6 +\left( \alpha^{-2}(1-\alpha)^{-2}\right) t^2 +  \left(\sigma_4+\frac{\sigma_3}{\alpha(1-\alpha)}+\frac{1}{\alpha^4(1-\alpha)^3}\right)t^4 \right)\]
	for some $d_1$, $d_2$. The contribution of $d_1 t$ and $d_2 t^3$ is $0$ because $e^{-\sigma_2 t^2/2}t^{2i+1}$ is odd.
	By Lemma~\ref{lem:hild2}, we have $\alpha, 1-\alpha\gg_{\varepsilon} 1/\log y$. Performing the change of variables $\sigma_2 t^2 = s^2$, and noting that
	\[ \frac{1}{\sqrt{2\pi}} \int_{-R}^{R} e^{-v^2/2}{\rm d}v =1+O(\exp(-cR^2)), \qquad \frac{1}{\sqrt{2\pi}} \int_{\RR} e^{-v^2/2}v^{2k} {\rm d}v \ll_k 1,\]
	it follows that
	\[ \int_{-\frac{1}{u^{1/3}\log y}}^{\frac{1}{u^{1/3}\log y}} \frac{\zeta_1(\alpha+it,y)x^{\alpha+it}\zeta(\alpha-1+it)H_y(\alpha+it)}{(\alpha+it)(\alpha+it-1)}{\rm d}t = \sqrt{\frac{2\pi}{\sigma_2}}\frac{\zeta_1(\alpha,y)x^{\alpha}\zeta(\alpha-1)H_y(\alpha)}{\alpha(\alpha-1)}(1+O_{\varepsilon}(u^{-1})).\]
	Plugging this in Proposition~\ref{prop:saddle prepare} we obtain
	\[ I_{\alpha}(x,y) = \frac{1}{\sqrt{2\pi\sigma_2}}\frac{\zeta_1(\alpha,y)x^{\alpha}\zeta(\alpha-1)H_y(\alpha)}{\alpha(\alpha-1)}(1+O_{\varepsilon}(u^{-1})) + O\left(\frac{(\log y)^2\Psi_{\mu^2}(x,y)}{ T^{\alpha}}\right),\]
for the $T$ in the statement of Proposition~\ref{prop:saddle prepare}. 	In \cite[Thm.~2.1]{LaB}, it is shown that
\[ \Psi_{\mu^2}(x,y) = \frac{x^{\alpha}\zeta_1(\alpha,y)}{\alpha\sqrt{2\pi \sigma_2}} (1+O_{\varepsilon}(u^{-1}))\]
if $y \ge (2+\varepsilon)\log x$.
	The last two estimates imply the desired conclusion.
\end{proof}
\section{Conclusion of proofs}\label{sec:final}
\subsection{Proof of Theorem~\ref{thm:mt}}
\subsubsection{Proof of 1st part of Theorem~\ref{thm:mt}}
Let $u=\log H/\log y$ and $u'=\log (H/2)/\log y$. We establish \eqref{eq:M2} in slightly stronger form: we show that if $y \ge \exp((\log \log H)^{5/3+\varepsilon})$, then
	\begin{equation}\label{eq:MHstr} M(H,y) = H \Prob_y\lambda(u)(1+O_{\varepsilon}(E))
	\end{equation}
holds where $E=(u+1)^2\log(u+2)/\log H$ if $y \ge \exp((\log H)^{1/2}\log \log H)$ and $E = 1/u + \log(u+2)/\log y$ otherwise. If $y \ge \exp((\log H)^{1/2}\log \log H)$ we apply \eqref{eq:gytail}, Lemma~\ref{lem:tailbnds} and \eqref{eq:ygreater} with $x=H/2$ to obtain 
\begin{multline*}
\sum_{r \ge 1}  g_y(r)\left\{ \frac{H}{2r}\right\} \left(1-\left\{ \frac{H}{2r}\right\}\right) \\= \frac{H}{2} \left( \prod_{2<p \le y}\left(1+\frac{1}{p-2}\right) \left(\lambda(u')\left( 1+O\left( \frac{(u'+1)^2\log(u'+2)}{\log H}\right)\right)\right) + O(\rho(u'))\right).
\end{multline*}
We multiply the last equation by $\prod_{2<p \le y} (1- 2/p)$ and use $\rho(u')\ll \lambda(u')  \log (u'+2)$ which follows from \eqref{eq:lambasym}. In this way we obtain \eqref{eq:MHstr} with $\lambda(u')$ instead of $\lambda(u)$. It now suffices to show
\begin{equation}\label{eq:lambdadiff2}
	\lambda(u') = \lambda(u)\left(1+ O\left(\frac{\log(u+2)}{\log y}\right)\right),
\end{equation}
which we demonstrate in the wider range $y \ge \log H$. If $H\le y$ \eqref{eq:lambdadiff2} follows from \eqref{eq:lam01}. Otherwise we use \eqref{eq:deflam} and \eqref{eq:lambasym} to deduce that
\begin{equation}\label{eq:lambdadiff}
	\frac{\lambda(u-t)-\lambda(u)}{\lambda(u)}\ll \frac{ \rho(u-t)}{\rho(u)}t\log (u+2)
\end{equation}
holds for $u\ge t>0$. By \cite[Lem.~1]{Hildebrand1984}, if $1 \ge t>0$ and $t \ll 1/\log (u+2)$, then $\rho(u-t)/\rho(u) \ll 1$. Thus \eqref{eq:lambdadiff} implies \eqref{eq:lambdadiff2}.

Next consider $\exp((\log H)^{1/2}\log \log H) \ge y \ge \exp((\log \log H)^{5/3+\varepsilon})$. Corollary~\ref{cor:ic} relates $M(H,y)$ to a contour integral estimated in Proposition~\ref{prop:i is psi}, and we obtain 
\begin{equation}\label{eq:MHpsi2}
M(H,y) = 2\prod_{2 <p \le y} \left(1-\frac{2}{p}\right) \frac{H_y(\alpha)\zeta(\alpha-1)}{\alpha-1}\Psi_{\mu^2}(H/2,y)(1+O_{\varepsilon}(1/u))
\end{equation}
where $\alpha=\alpha(H/2,y)$ is defined in \eqref{eq:alphadef} and $H_y$ is defined in  terms of an Euler product in \eqref{eq:HyEuler}. We now simplify: we relate $\Psi_{\mu^2}(H/2,y)$ to $\Psi(H/2,y)/\zeta(2)$ via \eqref{eq:ivict}, $\Psi(H/2,y)$ to $(H/2)\rho(u')$ via \eqref{eq:psirho}, $\rho(u')$ to $\lambda(u')$ via \eqref{eq:lambasym} and $\lambda(u')$ to $\lambda(u)$ via \eqref{eq:lambdadiff2}. We write $\zeta(\alpha-1)$ as $\zeta(0)+O(1-\alpha)$, the $(1+2^{-\alpha})^{-1}$ term in $H_y(\alpha)$ as $2/3+O(1-\alpha)$, and estimate $1-\alpha$ using the first part of Lemma~\ref{lem:hild2}. Furthermore we write $\zeta(2)$ as $\prod_{p \le y}(1-1/p)^{-1}(1+1/p)^{-1} (1+O(1/y))$. We obtain
\begin{equation}\label{eq:Mpre}
 M(H,y) =  H\Prob_y \lambda(u) \frac{e^{\gamma}\log y}{2}\prod_{2 <p \le y} \left(1-\frac{2}{p(1+p^{-\alpha})}\right)\left(1+\frac{1}{p}\right)\left(1+O_{\varepsilon}\left(\frac{1}{u}+\frac{\log (u+2)}{\log y}\right)\right).
\end{equation}
By \eqref{eq:mertens3}, $e^{\gamma}\log y /2 = (1+O(1/\log y))\prod_{2<p \le y}p/(p-1)$, simplifying \eqref{eq:Mpre} further as
\begin{equation}\label{eq:Prod}
 M(H,y) =  H\Prob_y \lambda(u) \prod_{2 <p \le y} \left(1+\frac{2(1-p^{\alpha-1})}{(1-1/p)(p^{1+\alpha}+p)}\right)\left(1+O_{\varepsilon}\left(\frac{1}{u}+\frac{\log (u+2)}{\log y}\right)\right).
\end{equation}
We shall show the product in the right-hand side of \eqref{eq:Prod} is acceptable. We express it as
\[ \prod_{2 <p \le y} \left(1+\frac{2(1-p^{\alpha-1})}{(1-1/p)(p^{1+\alpha}+p)}\right) = \exp\left( O\left(\sum_{p\le y} \frac{1-p^{\alpha-1}}{p^{1+\alpha}}\right)\right).\]
The expression in the exponent is 
\[\ll \sum_{\log p \le 1/(1-\alpha)} \frac{(1-\alpha)\log p}{p^{1+\alpha}} + \sum_{\log p >1/(1-\alpha)} \frac{1}{p^{1+\alpha}} \ll \alpha^{-1} (1-\alpha + (\exp(1/(1-\alpha))-1)^{-\alpha}).\]
By Lemma~\ref{lem:hild2}, this bound is $\ll 1-\alpha \ll \log (u+2)/\log y$.
\subsubsection{Proof of 2nd part of Theorem~\ref{thm:mt}}
Let $u=\log H/\log y$ and $u'=\log(H/2)/\log y$. We establish \eqref{eq:M5} in slightly stronger form: we show that if $H \ge y \ge (\log H)^{2+\varepsilon}$ then
\begin{equation}\label{eq:M5form}
	M(H,y)= \frac{e^{-\gamma}}{\xi(u)}\Prob_y \Psi(H,y)(1+O_{\varepsilon}(E + 1/u))
\end{equation}
for the same $E$ as in \eqref{eq:MHstr}. First consider $H \ge y \ge \exp((\log H)^{1/2}\log \log H)$. In this range, \eqref{eq:M5form} follows from \eqref{eq:MHstr}  once we replace $\lambda(u)$ with $e^{-\gamma}\rho(u)/\xi(u)$ using \eqref{eq:lambasym} and $H\rho(u)$ with $\Psi(H,y)$ using \eqref{eq:psirho}.

Next consider $\exp((\log H)^{1/2}\log \log H) \ge y \ge (\log H)^{2+\varepsilon}$. We shall simplify \eqref{eq:MHpsi2} in order to obtain \eqref{eq:M5form}. We relate $\Psi_{\mu^2}(H/2,y)$ to $\Psi(H/2,y)/\zeta(2)$ via \eqref{eq:ivict} and write $\zeta(\alpha-1)$ as $\zeta(0)+O(1-\alpha)$, the $(1+2^{-\alpha})^{-1}$ term in $H_y(\alpha)$ as $2/3+O(1-\alpha)$, and estimate $1-\alpha$ using the first part of Lemma~\ref{lem:hild2}. Furthermore we write $\zeta(2)$ as $\prod_{p \le y}(1-1/p)^{-1}(1+1/p)^{-1} (1+O(1/y))$. We obtain from \eqref{eq:MHpsi2}
\begin{equation}\label{eq:Mpre2}
	M(H,y) =  \frac{\log y}{\xi(u')}\Prob_y\Psi(H/2,y) \prod_{2 <p \le y} \left(1-\frac{2}{p(1+p^{-\alpha})}\right)\left(1+\frac{1}{p}\right)\left(1+O_{\varepsilon}\left(\frac{1}{u}+\frac{\log (u+2)}{\log y}\right)\right).
\end{equation}
By \eqref{eq:mertens3}, $\log y$ is $(2+O(1/\log y))e^{-\gamma}\prod_{2<p \le y}p/(p-1)$, simplifying \eqref{eq:Mpre2} further as
\[ M(H,y) =  \frac{e^{-\gamma}}{\xi(u')} \Prob_y 2\Psi(H/2,y)  \prod_{2 <p \le y} \left(1+\frac{2(1-p^{\alpha-1})}{(1-1/p)(p^{1+\alpha}+p)}\right)\left(1+O_{\varepsilon}\left(\frac{1}{u}+\frac{\log (u+2)}{\log y}\right)\right).\]
As in the proof of \eqref{eq:MHstr}, the product is $1+O_{\varepsilon}(\log(u+2)/\log y))$. To conclude, we need to simplify $\xi(u')$ and $\Psi(H/2,y)$. Since $\xi'(t)\ll 1/t$ \cite[Lem.~1]{Hildebrand1984}, we have $\xi(u')=\xi(u) (1+O_{\varepsilon}(1/\log H))$ by the mean value theorem. To treat $\Psi(H/2,y)$ we use \cite[Thm.~3]{Hildebrand1986} to find \[ \Psi(H/2,y) =\Psi(H,y)2^{-\alpha'}(1+O(1/u))\] where $\alpha'$ satisfies $\sum_{p \le y} \log p/(p^{\alpha'}-1)=\log (H/2)$. By \cite[Lem.~2]{Hildebrand1986}, $2^{1-\alpha'} = 1+O_{\varepsilon}(\log(u+2)/\log y)$.
\subsubsection{Proof of 3rd part of Theorem~\ref{thm:mt}}
Let $u=\log H/\log y$ and $u'=\log(H/2)/\log y$. We establish \eqref{eq:M4} in slightly stronger form: we show that if $H \ge y \ge (2+\varepsilon)\log H$ then
\begin{equation}\label{eq:M5form2}
	M(H,y)=\frac{2\zeta(\alpha-1)}{\alpha-1} \prod_{p \le y} \left(1-\frac{2}{p(1+p^{-\alpha})} \right) \Psi_{\mu^2}(H,y)(1+O_{\varepsilon}(E'))
\end{equation}
where $E'=1/\sqrt{u}$ if $y\ge (\log H)^2$ and $E'=\exp(-(\log y)^{1/2-2\varepsilon})$ otherwise.

Suppose first that $u \ge (\log \log (2y))^6$. Corollary~\ref{cor:ic} connects $M(H,y)$ with a contour integral which is estimated in Proposition~\ref{prop:i is psi}, and we obtain
\begin{equation}\label{eq:psiest}
	M(H,y)= 2 \prod_{2<p \le y} \left(1- \frac{2}{p}\right)\frac{H_y(\alpha)\zeta(\alpha-1)}{\alpha-1} \Psi_{\mu^2}(H/2,y)(1 + O_{\varepsilon}(1/u +  \exp(-\alpha (\log y)^{3/2-\varepsilon}))),
\end{equation}
where $\alpha=\alpha(H/2,y)$ is defined in \eqref{eq:alphadef} and estimated in Lemma~\ref{lem:hild2}, and $H_y$ is defined in \eqref{eq:defh}. Inputting the Euler product for $H_y$ given in \eqref{eq:HyEuler}, and using the estimate
\[ \Psi_{\mu^2}(H/2,y) = \Psi_{\mu^2}(H,y) 2^{-\alpha} (1+O_{\varepsilon}( 1/\sqrt{u}))\]
given in \cite[Cor.~2.4]{LaB}, yields \eqref{eq:M5form2}.

Suppose now that $1\le u \le (\log \log (2y))^6$. We claim \eqref{eq:M5form2} follows from \eqref{eq:M5form}. Indeed, we can write 
\[ \Psi(H,y) = \zeta(2) \Psi_{\mu^2}(H,y)\left(1+O\left(\frac{\log \log y}{\log y}\right)\right)\]
by \eqref{eq:ivict}, and express 
\[\frac{1}{\xi(u)} = \frac{1+O(1/\log H)}{(1-\alpha)\log y}\]
using Lemma~\ref{lem:hild2} to deduce from \eqref{eq:M5form} that
 \begin{equation}\label{eq:M5form3}
	M(H,y)= \frac{\zeta(2)e^{-\gamma}}{(1-\alpha)\log y} \prod_{p \le y} \left(1-\frac{1}{p}\right) \Psi_{\mu^2}(H,y) (1+O(1/\sqrt{u}))
\end{equation}
holds. Using \eqref{eq:mertens3} we replace $e^{-\gamma}/\log y$ with $\prod_{p\le y}(1-1/p) ( 1+O(1/\log y))$. We write $\zeta(2) = \prod_{p \le y}(1-1/p^2)^{-1} ( 1+O(1/y))$. This transforms \eqref{eq:M5form3} into
\[	M(H,y)= \frac{1}{1-\alpha} \prod_{p \le y} \frac{p-1}{p+1}\Psi_{\mu^2}(H,y) (1+O(1/\sqrt{u})).\]
We are done since, for $1 \le u \le (\log \log (2y))^6$, $-2\zeta(\alpha-1) = 1 + O(\log (u+1)/\log y)$ holds (see Lemma~\ref{lem:hild2}) as well as
\[ \prod_{p \le y} \left(1- \frac{2}{p(1+p^{-\alpha})}\right) = \prod_{p \le y} \frac{p-1}{p+1} (1+O(1/\sqrt{u}))\]
in exactly the same way we estimated the product in \eqref{eq:Prod}.
\subsubsection{Proof of 4th part of Theorem~\ref{thm:mt}}
We establish \eqref{eq:last}. 
If $y \ge H$, then $u \le 1$, so $\log (u+2) \asymp 1$ and $\lambda(u) \asymp 1$. The estimate \eqref{eq:last} is a consequence of \eqref{eq:MHstr} and \eqref{eq:psirho}--\eqref{eq:ivict}. If $u$ is bounded, the same argument still establishes \eqref{eq:last}.

If $H \ge y \ge (\log H)^3$ and $u$ is sufficiently large, \eqref{eq:last} is a consequence of \eqref{eq:M5form} and \eqref{eq:ivict}.

If $(\log H)^3 \ge y \ge (\log H)^{1+\varepsilon}$ and $H \ge C_{\varepsilon}$, then $c_{\varepsilon}\le \alpha \le 1-c$ for some $c,c_{\varepsilon}>0$ by Lemma~\ref{lem:hild2}. Now \eqref{eq:last} is in \eqref{eq:M5form2} once we observe $2\zeta(\alpha-1)/(\alpha-1) \asymp 1$ and $\prod_{p \le y} (1-2/(p(1+p^{-\alpha})) \asymp_{\varepsilon} (\log y)^{-2} \asymp \Prob_y /\log (u+2)$ in the considered range.
\subsection{Proof of Theorem~\ref{thm:ranges}} 
Fix $\varepsilon>0$ and suppose $y \ge (2+\varepsilon)\log H$. Using \eqref{eq:psimu2crude}, Theorem~\ref{thm:mt} implies
\begin{equation}\label{eq:hild crude}
	M(H,y) \gg_{\varepsilon}\Prob_y H^{1-\frac{\log \log H}{\log y}+o(1)}=\Prob_y H^{1-a+o(1)}
\end{equation}
where $a$ is defined in Theorem~\ref{thm:ranges}.
By Proposition~\ref{prop:sieveapproach} and \eqref{eq:hild crude}, a sufficient condition for $V(X,H,y) \sim M(H,y)$ to hold is that there is $u \ge 1$ such that
\[ H^2 \Prob_y  u^{-u(1-\delta)}, \frac{H^2\Psiomega(y^u,y)}{X} \ll \Prob_y H^{1-a-\delta}\]
both hold for some $\delta>0$. The inequality $H^2 \Prob_y u^{-u(1-\delta)} \ll \Prob_y H^{1-a-\delta}$ simplifies to $u^{u(1-\delta)} \gg \Prob_y H^{1+a+\delta}$, so taking
\[ u = \frac{(1+a)\log H}{\log \log H}(1+3\delta)\]
we see that it is satisfied. This means $V(X,H,y) \sim M(H,y)$ shall hold if, for this choice of $u$, we will have
\begin{equation}\label{eq:ucond}
	\frac{H^2\Psiomega(y^u,y)}{X} \ll \Prob_y H^{1-a-\delta}.
\end{equation}
Note that \[ \frac{\log \log (y^u)}{\log y} = \frac{\log \log H}{\log y} + o(1) = a+o(1),\] 
and by the divisor bound, \eqref{eq:psimu2crude} gives
\[ \Psiomega(y^u,y) = y^{u\left( 1- a+o(1)\right)} = H^{\frac{(1+a)}{a}(1+3\delta)\left( 1-a+o(1)\right)}=H^{\frac{1-a^2+o(1)}{a}(1+3\delta)}.\]
Our condition \eqref{eq:ucond} becomes, after isolating the powers of $H$ and taking logarithms,
\[\left(1+a+\delta+\frac{1-a^2+o(1)}{a}(1+3\delta))\right)\log H \le \log X -\log \log y+O(1).\]
Multiplying and dividing the left-hand side by $a=\log \log H/\log y$, and neglecting $\log \log y$ (since it is $o(\log X)$) gives the desired result, that is, $V(X,H,y) \sim M(H,y)$ holds when \eqref{eq:sievecond} is satisfied.

\subsection*{Acknowledgements}
We thank the referee for useful comments and suggestions that improved the present paper.
This project has received funding from the European Research Council (ERC) under the European Union's Horizon 2020 research and innovation programme (grant agreement No 851318).

	\bibliographystyle{abbrv}
	\bibliography{references}

\begin{thebibliography}{10}

\bibitem{DeBruijn1950}
N.~G. de~Bruijn.
\newblock On the number of uncancelled elements in the sieve of {E}ratosthenes.
\newblock {\em Nederl. Akad. Wetensch., Proc.}, 53:803--812 = Indagationes
  Math. 12, 247--256 (1950), 1950.

\bibitem{deBruijn1966}
N.~G. de~Bruijn.
\newblock On the number of positive integers {$\leq x$} and free prime factors
  {$>y$}. {II}.
\newblock {\em Nederl. Akad. Wetensch. Proc. Ser. A 69=Indag. Math.},
  28:239--247, 1966.

\bibitem{dBL}
N.~G. de~Bruijn and J.~H. van Lint.
\newblock Incomplete sums of multiplicative functions. {I}.
\newblock {\em Indag. Math.}, 26:339--347, 1964.
\newblock Nederl. Akad. Wetensch. Proc. Ser. A 67.

\bibitem{LaB}
R.~de~la Bret\`eche and G.~Tenenbaum.
\newblock Sur les lois locales de la r\'{e}partition du {$k$}-i\`eme diviseur
  d'un entier.
\newblock {\em Proc. London Math. Soc. (3)}, 84(2):289--323, 2002.

\bibitem{Friedlander1991}
J.~Friedlander, A.~Granville, A.~Hildebrand, and H.~Maier.
\newblock Oscillation theorems for primes in arithmetic progressions and for
  sifting functions.
\newblock {\em J. Amer. Math. Soc.}, 4(1):25--86, 1991.

\bibitem{Friedlander2010}
J.~Friedlander and H.~Iwaniec.
\newblock {\em Opera de cribro}, volume~57 of {\em American Mathematical
  Society Colloquium Publications}.
\newblock American Mathematical Society, Providence, RI, 2010.

\bibitem{Goldston1987}
D.~A. Goldston and H.~L. Montgomery.
\newblock Pair correlation of zeros and primes in short intervals.
\newblock In {\em Analytic number theory and {D}iophantine problems
  ({S}tillwater, {OK}, 1984)}, volume~70 of {\em Progr. Math.}, pages 183--203.
  Birkh\"{a}user Boston, Boston, MA, 1987.

\bibitem{Gorodetsky2021}
O.~Gorodetsky, K.~Matom\"{a}ki, M.~Radziwi{\l}{\l}, and B.~Rodgers.
\newblock On the variance of squarefree integers in short intervals and
  arithmetic progressions.
\newblock {\em Geom. Funct. Anal.}, 31(1):111--149, 2021.

\bibitem{Granville1989}
A.~Granville.
\newblock On positive integers {$\leq x$} with prime factors {$\leq t\log x$}.
\newblock In {\em Number theory and applications ({B}anff, {AB}, 1988)}, volume
  265 of {\em NATO Adv. Sci. Inst. Ser. C: Math. Phys. Sci.}, pages 403--422.
  Kluwer Acad. Publ., Dordrecht, 1989.

\bibitem{Granville2007}
A.~Granville and K.~Soundararajan.
\newblock An uncertainty principle for arithmetic sequences.
\newblock {\em Ann. of Math. (2)}, 165(2):593--635, 2007.

\bibitem{HR1}
H.~Halberstam and H.-E. Richert.
\newblock Brun's method and the fundamental lemma. {I}.
\newblock {\em Acta Arith.}, 24:113--133, 1973.

\bibitem{HR2}
H.~Halberstam and H.-E. Richert.
\newblock Brun's method and the fundamental lemma. {II}.
\newblock {\em Acta Arith.}, 27:51--59, 1975.

\bibitem{Hall1982}
R.~R. Hall.
\newblock Squarefree numbers on short intervals.
\newblock {\em Mathematika}, 29(1):7--17, 1982.

\bibitem{Hausman1973}
M.~Hausman and H.~N. Shapiro.
\newblock On the mean square distribution of primitive roots of unity.
\newblock {\em Comm. Pure Appl. Math.}, 26:539--547, 1973.

\bibitem{Hildebrand1984}
A.~Hildebrand.
\newblock Integers free of large prime factors and the {R}iemann hypothesis.
\newblock {\em Mathematika}, 31(2):258--271 (1985), 1984.

\bibitem{Hildebrand1985}
A.~Hildebrand.
\newblock Integers free of large prime divisors in short intervals.
\newblock {\em Quart. J. Math. Oxford Ser. (2)}, 36(141):57--69, 1985.

\bibitem{Hildebrand19862}
A.~Hildebrand.
\newblock On the number of positive integers {$\leq x$} and free of prime
  factors {$>y$}.
\newblock {\em J. Number Theory}, 22(3):289--307, 1986.

\bibitem{Hildebrand1986}
A.~Hildebrand and G.~Tenenbaum.
\newblock On integers free of large prime factors.
\newblock {\em Trans. Amer. Math. Soc.}, 296(1):265--290, 1986.

\bibitem{Ivic1986}
A.~Ivi\'{c} and G.~Tenenbaum.
\newblock Local densities over integers free of large prime factors.
\newblock {\em Quart. J. Math. Oxford Ser. (2)}, 37(148):401--417, 1986.

\bibitem{Kuperberg2020}
V.~Kuperberg, B.~Rodgers, and E.~Roditty-Gershon.
\newblock Sums of singular series and primes in short intervals in algebraic
  number fields.
\newblock {\em Ramanujan J.}, 58(2):291--317, 2022.

\bibitem{Montgomery2010}
H.~L. Montgomery.
\newblock The combinatorics of moment calculations.
\newblock {\em Hardy-Ramanujan J.}, 33:2--22, 2010.

\bibitem{Montgomery2002}
H.~L. Montgomery and K.~Soundararajan.
\newblock Beyond pair correlation.
\newblock In {\em Paul {E}rd\H{o}s and his mathematics, {I} ({B}udapest,
  1999)}, volume~11 of {\em Bolyai Soc. Math. Stud.}, pages 507--514. J\'{a}nos
  Bolyai Math. Soc., Budapest, 2002.

\bibitem{Montgomery2004}
H.~L. Montgomery and K.~Soundararajan.
\newblock Primes in short intervals.
\newblock {\em Comm. Math. Phys.}, 252(1-3):589--617, 2004.

\bibitem{Montgomery1986}
H.~L. Montgomery and R.~C. Vaughan.
\newblock On the distribution of reduced residues.
\newblock {\em Ann. of Math. (2)}, 123(2):311--333, 1986.

\bibitem{Montgomery2007}
H.~L. Montgomery and R.~C. Vaughan.
\newblock {\em Multiplicative number theory. {I}. {C}lassical theory},
  volume~97 of {\em Cambridge Studies in Advanced Mathematics}.
\newblock Cambridge University Press, Cambridge, 2007.

\bibitem{Naimi1988}
M.~Naimi.
\newblock Les entiers sans facteurs carr\'{e} {$\leq x$} dont leurs facteurs
  premiers {$\leq y$}.
\newblock In {\em Groupe de travail en th\'{e}orie analytique et
  \'{e}l\'{e}mentaire des nombres, 1986--1987}, volume~88 of {\em Publ. Math.
  Orsay}, pages 69--76. Univ. Paris XI, Orsay, 1988.

\bibitem{Saffari1977}
B.~Saffari and R.~C. Vaughan.
\newblock On the fractional parts of {$x/n$} and related sequences. {II}.
\newblock {\em Ann. Inst. Fourier (Grenoble)}, 27(2):v, 1--30, 1977.

\bibitem{selberg}
A.~Selberg.
\newblock {\em Collected papers. {V}ol. {I}}.
\newblock Springer-Verlag, Berlin, 1989.

\bibitem{Tenenbaum2015}
G.~Tenenbaum.
\newblock {\em Introduction to analytic and probabilistic number theory},
  volume 163 of {\em Graduate Studies in Mathematics}.
\newblock American Mathematical Society, Providence, RI, third edition, 2015.
\newblock Translated from the 2008 French edition by Patrick D. F. Ion.

\bibitem{Tenenbaum2003}
G.~Tenenbaum and J.~Wu.
\newblock Moyennes de certaines fonctions multiplicatives sur les entiers
  friables.
\newblock {\em J. Reine Angew. Math.}, 564:119--166, 2003.

\bibitem{Tenenbaum2008}
G.~Tenenbaum and J.~Wu.
\newblock Moyennes de certaines fonctions multiplicatives sur les entiers
  friables. {III}.
\newblock {\em Compos. Math.}, 144(2):339--376, 2008.

\bibitem{Titchmarsh}
E.~C. Titchmarsh.
\newblock {\em The theory of the {R}iemann zeta-function}.
\newblock The Clarendon Press, Oxford University Press, New York, second
  edition, 1986.
\newblock Edited and with a preface by D. R. Heath-Brown.

\bibitem{vLR}
J.~H. van Lint and H.-E. Richert.
\newblock {\"U}ber die {Summe} {{\(\sum_{n\leq
  x,p(n)<y}\mu^2(n)/\varphi(n)\)}}.
\newblock {\em Nederl. Akad. Wet., Proc., Ser. A}, 67:582--587, 1964.

\end{thebibliography}

	\Addresses
\end{document}